\documentclass[11pt]{article}
\usepackage{xcolor} 

\usepackage{amssymb}   
\usepackage{amsthm}    
\usepackage{amsmath}   
\usepackage{stmaryrd}  
\usepackage{mathrsfs}  
\usepackage{graphicx}
\usepackage{algorithm}
\usepackage{algpseudocode}
\usepackage{url}

\newlength{\defbaselineskip}
\newcommand{\setlinespacing}[1]%
           {\setlength{\baselineskip}{#1 \defbaselineskip}}

\theoremstyle{plain}
\newtheorem{thm}{Theorem}[section]
\newtheorem{cor}[thm]{Corollary}
\newtheorem{lem}[thm]{Lemma}
\newtheorem{prop}[thm]{Proposition}
\newtheorem{exam}[thm]{Example}

\theoremstyle{definition}

\newtheorem{ass}{Assumption}[section]
\newtheorem{rmk}{Remark}[section]

\newcommand{\eps}{\varepsilon}

\newcommand{\cL}{\mathcal{L}}

\newcommand{\bR}{\mathbb{R}}

\newcommand{\ptl}{\partial}

\newcommand{\E}{\mathbb{E}}
\newcommand{\Pcal}{\mathcal{P}_2(\bR^d)}
\newcommand{\Lions}{\partial_\mu}

\makeatletter\@addtoreset{equation}{section} \makeatother
 \allowdisplaybreaks
\begin{document}

\title{Stochastic Control Methods for Optimization}

\author{Jinniao Qiu\footnotemark[1]\footnotemark[2]  
}
\date{
 First version: January 1, 2026; Revised version: April 18, 2026.}
\footnotetext[1]{This work was partially supported by the National Science and Engineering Research Council of Canada (NSERC). }
\footnotetext[2]{Department of Mathematics \& Statistics, University of Calgary, 2500 University Drive NW, Calgary, AB T2N 1N4, Canada. \textit{E-mail}: \texttt{jinniao.qiu@ucalgary.ca} (J. Qiu)
.}
 
\maketitle

\begin{abstract}
In this work, we investigate a stochastic control framework for global optimization over both Euclidean spaces and 
the Wasserstein space of probability measures, where the objective function may be non-convex and/or non-differentiable. 
In the Euclidean setting, the original minimization problem is approximated by a family of 
regularized stochastic control problems; using dynamic programming, we analyze the associated Hamilton--Jacobi--Bellman 
equations and obtain tractable representations via the Cole--Hopf transformation and the Feynman--Kac formula. 
For optimization over probability measures, we formulate a regularized mean-field control problem 
characterized by a master equation, and further approximate it by controlled $N$-particle systems. 
We establish that, as the regularization parameter tends to zero (and as the particle number tends to infinity for the optimization over probability measures), 
the value of the control problem converges to the global minimum of the original objective. 
Building on the resulting probabilistic representations, we propose the Monte Carlo-based numerical schemes that are derivative-free due to the utilization of the Bismut-Elworthy-Li formula and 
numerical experiments are reported to illustrate the effectiveness of the methods and to support the theoretical convergence rates.
\end{abstract}
\smallskip 



\section{Introduction}

In this work, we are concerned with the following classical global optimization problem:
\begin{equation}\label{opt-prob}
    \min_{x\in \mathcal X} G(x),
\end{equation}
where $\mathcal X$ is a finite-dimensional space $\bR^d$ or the space of probability measures 
and the objective function $G:\mathcal X \rightarrow \bR$ may be non-convex and/or non-differentiable.
Optimization problems of the form \eqref{opt-prob} arise in a wide range of applications, including operational research, 
machine learning, artificial intelligence, finance, engineering design, economics, and operations research. 
Traditional optimization methods, such as gradient descent and Newton's method, often struggle with non-convex or 
non-differentiable functions due to the presence of multiple local minimas and the lack of gradient information; see 
 \cite{danilova2022recent,velasco2024literature}.
To address these challenges, we propose and study stochastic control methods (SCM) to solve the optimization problem \eqref{opt-prob}. 

\subsection{Optimization over finite-dimensional spaces: $\mathcal X = \bR^d$} 

\subsubsection{Stochastic control formulations of problem \eqref{opt-prob}}
We introduce the following stochastic control problem:
\begin{equation}\label{sto-control-prob}
    \min_{\theta\in \Theta} \mathbb E\left[ G(X_1^{0,x_0;\theta})\right],
\end{equation}
subject to the controlled stochastic differential equation (SDE):
\begin{equation}\label{controlled-SDE}
    X_t^{0,x_0;\theta} = x_0+\int_0^t \theta_s ds + W_t, \quad t\in[0,1].
\end{equation}
Here and below, we take terminal time $T=1$ and the $\bR^d$-valued initial value $x_0$ and the $d$-dimensional Wiener process $(W_t)_{t\geq 0}$ are 
defined on a complete filtered probability space $(\Omega, \mathcal{F},(\mathcal{F}_t)_{t\geq 0}, \mathbb{P})$ so that 
the filtration $(\mathcal{F}_t)_{t\geq 0}$ is generated by $x_0$ and $(W_t)_{t\geq 0}$ and augmented by all the 
$\mathbb P$-null sets in $\mathcal{F}$. In particular, $x_0$ is $\mathcal F_0$-measurable and independent of the Wiener process $(W_t)_{t\geq 0}$. 
The admissible control $\theta$ is an $(\mathcal{F}_t)_{t\geq 0}$-adapted process, and the set of admissible controls, denoted by $\Theta$, consists of all such admissible control processes $\theta$  
 satisfying $\E\left[\left(\int_0^1|\theta_s|ds\right)^2\right]<\infty$.

Suppose $\xi\in\bR^d$ is a minimizer of \eqref{opt-prob}, i.e., $G(\xi)=\min_{x\in \bR^d} G(x)$. Then we may take the feedback control $\bar\theta_t=\frac{\xi-X^{0,x_0;\bar\theta}_t}{1-t} \textbf{1}_{[0,1)}(t)$ and the Brownian bridge theory 
(see \cite{karatzas1998brownian}) implies that $X_1^{0,x_0;\theta}=\xi$ almost surely (a.s. for short).  This yields the equivalence of the optimization problem \eqref{opt-prob} and the stochastic control problem \eqref{sto-control-prob}.
Just note that the optimal control $\bar\theta$ is not unique in $\Theta$ since we may take $\theta_t\equiv 0$ for $t\in[0,\frac{1}{2})$ and construct the Brownian bridge over the time interval $[\frac{1}{2},1]$ from $X^{0,x_0;\theta}_{1/2}=x_0$ to $\xi$. 
This nonuniqueness incurs difficulties in characterizing the optimal control. To address this issue, we consider the following regularized stochastic control problem:
\begin{equation}\label{reg-sto-control-prob}
    \min_{\theta\in \Theta} \mathbb E\left[ G(X_1^{0,x_0;\theta}) + \frac{\eps}{2}\int_0^1 |\theta_t|^2 dt\right],
\end{equation}
subject to the controlled SDE \eqref{controlled-SDE}, where $\eps\in(0,e^{-1})$ is a regularization parameter.
The additional quadratic term in the cost functional penalizes large control values and encourages smoother control strategies.
Under suitable conditions on the objective function $G$, this regularization helps to facilitate the characterization of the unique optimal control, which also paves the way for numerical approximations.


\subsubsection{Recall a standard resolution of the regularized stochastic control problem \eqref{reg-sto-control-prob}}
Indeed, we may define the cost functional associated with the regularized stochastic control problem \eqref{reg-sto-control-prob} as follows:
\begin{equation}\label{cost-functional}
    J^{\theta}_{\eps}(t,x) = \mathbb E\left[ G(X_1^{t,x;\theta}) + \frac{\eps}{2}\int_t^1 |\theta_s|^2 ds  \right],
\end{equation}
with $X_s^{t,x;\theta}=x+\int_t^s \theta_rdr + W_s-W_t$ for $s\in[t,1]$, and set the value function as
\begin{equation}\label{value-function}
    V_{\eps}(t,x) = \inf_{\theta\in \Theta} J^{\theta}_{\eps}(t,x).
\end{equation}
By the dynamic programming principle (\cite{fleming2006controlled,YZ99}), the value function $V_{\eps}$ satisfies the following Hamilton-Jacobi-Bellman (HJB) equation:
\begin{equation}\label{HJB-eqn}
    \left\{\begin{aligned}
        &-\ptl_t V_{\eps}(t,x) - \frac{1}{2}\Delta V_{\eps}(t,x)   + \frac{1}{2\eps}|\partial_x V_{\eps}(t,x)|^2 = 0, \quad (t,x)\in[0,1)\times\bR^d,\\
        &V_{\eps}(1,x) = G(x), \quad x\in\bR^d.
    \end{aligned}\right.
\end{equation} 
The quadratic term inspires us to use the Cole-Hopf transformation to convert the nonlinear HJB equation \eqref{HJB-eqn} into a linear backward heat equation.  Specifically, we define the following transformation:
\begin{equation}\label{Cole-Hopf-transformation}
    u(t,x) = e^{-\frac{1}{\eps}V_{\eps}(t,x)},
\end{equation}
which leads to the following linear PDE for $u(t,x)$:
\begin{equation}\label{PDE-u}
    \left\{\begin{aligned} 
        &-\ptl_t u(t,x) - \frac{1}{2}\Delta u(t,x) = 0, \quad (t,x)\in[0,1)\times\bR^d,\\
        &u(1,x) = e^{-\frac{1}{\eps}G(x)}, \quad x\in\bR^d.
    \end{aligned}\right.
\end{equation}
Then the Feynman-Kac formula implies that the solution to \eqref{PDE-u} can be represented as
\begin{equation}\label{FK-representation}
    u(t,x) = \mathbb E\left[ e^{-\frac{1}{\eps}G(W_1-W_t+x)}  \right]\, ,
\end{equation}
for $(t,x)\in[0,1]\times\bR^d$.  Combining \eqref{Cole-Hopf-transformation} and \eqref{FK-representation}, we obtain the following representation of the value function $V_{\eps}$:
\begin{equation}\label{value-function-representation}
    V_{\eps}(t,x) = -\eps \ln \mathbb E\left[ e^{-\frac{1}{\eps}G(W_1-W_t+x)}   \right].
\end{equation}

The preceding derivations are largely standard; we include them to establish a self-contained foundation for the developments in subsequent sections, in particular our treatment of optimization over spaces of probability measures. 
In stochastic control, the Cole--Hopf transformation can be traced back to 
Fleming \cite{fleming1978optimal}. For further applications and extensions, see, e.g., \cite{berneroptimal,fleming2006controlled,leger2021hopf} and the references therein.

\subsubsection{Main result on convergence analysis as $\eps\to 0$}
Solving the HJB equation \eqref{HJB-eqn} yields the optimal (feedback) control $\theta^*_{\eps}$ with the following characterization:
\begin{align}
    \theta^*_{\eps}\!\left(t,X_t^{0,x_0;\theta_{\eps}^*}\right)
    = -\frac{1}{\eps}\partial_x V_{\eps}\!\left(t,X_t^{0,x_0;\theta_{\eps}^*}\right)
    &= \frac{\partial_x \mathbb E\!\left[ \exp\!\left(-\frac{1}{\eps}G(W_1-W_t+x)\right) \right]}{\mathbb E\!\left[ \exp\!\left(-\frac{1}{\eps}G(W_1-W_t+x)\right) \right]}
    \bigg|_{x=X_t^{0,x_0;\theta_{\eps}^*}}
    \label{optimal-control-in-terms-of-V}\\
    &= \frac{\mathbb E\!\left[ \exp\!\left(-\frac{1}{\eps}G(W_1-W_t+x)\right)\,\frac{W_1-W_t}{1-t} \right]}{\mathbb E\!\left[ \exp\!\left(-\frac{1}{\eps}G(W_1-W_t+x)\right) \right]}
    \bigg|_{x=X_t^{0,x_0;\theta_{\eps}^*}},
    \label{optimal-control-characterization}
\end{align}
where from \eqref{optimal-control-in-terms-of-V} to \eqref{optimal-control-characterization} we have used the integration-by-parts formula (see \eqref{representation-gradient}) or so-called Bismut-Elworthy-Li formula (see, e.g., \cite{ElworthyLi1994}).

Substituting the optimal control characterization \eqref{optimal-control-characterization} into the controlled SDE \eqref{controlled-SDE}, we obtain the following closed-loop SDE for the optimal state process $X_t^{0,x_0;\theta_{\eps}^*}$:
\begin{equation}\label{closed-loop-SDE}
    X_t^{0,x_0;\theta_{\eps}^*} = x_0 + \int_0^t \frac{\mathbb E\left[ e^{-\frac{1}{\eps}G(W_1-W_s+x)} \frac{(W_1 - W_s)}{1-s}   \right]}{ \mathbb E\left[ e^{-\frac{1}{\eps}G(W_1-W_s+x)}   \right]}\bigg|_{x=X_s^{0,x_0;\theta_{\eps}^*}} ds + W_t, \quad t\in[0,1].
\end{equation}
Then using existing Monte-Carlo methods for SDEs (see, e.g., \cite{kloeden1992numerical}), 
we may simulate the optimal state process $X_t^{0,x_0;\theta_{\eps}^*}$. 
When $\eps$ is sufficiently small, the terminal value $X_1^{0,x_0;\theta_{\eps}^*}$ 
approximates a global minimizer of the original optimization problem \eqref{opt-prob}. 
The following theorem encapsulates the convergence analysis presented in Theorems \ref{thm-convergence} and \ref{thm-convergence-2}, providing an explicit error bound for the approximation of the global minimum.
\begin{thm}\label{thm-main-fDim}
Under certain assumptions on $G$, for each $\xi \in \arg\min_{x\in\bR^d}G(x)$, it holds that
    \begin{equation}
        0\leq \E\left[ V_{\eps}(0,x_0)\right]  - G(\xi) 
        \leq C \, \eps\, \ln\left(\frac{1}{\eps}\right)  ,  
    \end{equation}
    where the constant $C$ depends on dimension $d$, regularity of $G$, and $\mathbb E\left[|x_0-\xi|^2\right]$.
\end{thm}

 \subsection{Optimization over probability measures: $\mathcal X= \mathcal P_2(\bR^d)$} \label{section:opt-prob-measure}

Here we denote by $\Pcal$ the space of probability measures on $\bR^d$ with finite second moments, equipped with the $2$-Wasserstein metric.  
For simplicity, we take $\mathcal X = \Pcal$, and then  we restate the concerned minimization problem:
\begin{equation}
    \label{eq:global_min}
    \min_{\mu \in \Pcal} G(\mu),
\end{equation}
where $G: \Pcal \to \bR$ is a potentially non-convex and/or non-differentiable\footnote{Throughout this work, the differentiability, non-differentiability, derivatives, and convexity of functionals on probability measures are considered in the sense of Lions' derivatives and calculus on Wasserstein space.} functional, possibly including interaction terms.
Direct minimization of \eqref{eq:global_min} is difficult due to the infinite-dimensional and non-linear nature of the space and possible non-convexity and/or non-differentiability of the cost functional.
 Analogously, we formulate the following stochastic mean-field  control (MFC for short) problem:
\begin{equation}\label{reg-sto-control-prob-mf}
    \min_{\theta\in \Theta} G(\mathcal{L}(X_1))  + \mathbb E\left[   \frac{\eps}{2}\int_0^1 |\theta_t|^2 dt\right],
\end{equation}
subject to the controlled SDE \eqref{controlled-SDE}, where $\eps\in(0,\frac{1}{e})$ is a regularization parameter.


\subsubsection{The master equation and $N$-Particle Approximation}
For any square-integrable random variable $\zeta$, denote by $\cL(\zeta)$ its law/distribution. We define the cost functional 
 \begin{equation}\label{cost-functional-MF}
    J^{\theta}_{\eps}(t,\mathcal L (\xi)) =
    G(\mathcal{L}(X_1^{t,\xi;\theta}))  + 
     \mathbb  E\left[   \frac{\eps}{2}\int_t^1 |\theta_s|^2 ds  \right],
\end{equation}
and set the value function
\begin{equation}\label{value-function-MF}
    \mathcal V_{\eps}(t,\mu) = \inf_{\theta\in \Theta} J^{\theta}_{\eps}(t,\mu).
\end{equation}
Applying the dynamic programming principle yields (\cite{bensoussan2013mean,cardaliaguet2019master,carmona2015forward,carmona2018probabilistic}) the Master Equation, or so-called HJB equation on Wasserstein space:
\begin{equation}\label{Master-eq}
    \begin{cases}
    \displaystyle \partial_t \mathcal{V}_\eps - \frac{1}{2\eps} \int_{\bR^d} |\Lions \mathcal{V}_\eps|^2 d\mu + \frac{1}{2} \int_{\bR^d} \text{div}_x (\Lions \mathcal{V}_\eps) d\mu = 0, & (t, \mu) \in [0, 1) \times \Pcal \\
    \displaystyle \mathcal{V}_\eps(1, \mu) = G(\mu), & \mu \in \Pcal,
    \end{cases}
\end{equation}
where $\Lions \mathcal V_\eps$ denotes Lions' derivative. 
Solving this nonlinear equation directly is intractable; in particular the Cole-Hopf transformation working for finite-dimensional cases is not applicable herein.


To compute the solution, we approximate the measure $\mu$ by the empirical measure of $N$ particles $\mu^N = \frac{1}{N} \sum_{i=1}^N \delta_{x_i}$.
We define the $N$-particle terminal cost as the scaled total potential:
$$
U(\mathbf{x}) = N \cdot G\left(\frac{1}{N}\sum_{i=1}^N \delta_{x_i}\right).
$$
Let $v_{\eps}^N(t, \mathbf{x})$, for $(t, \mathbf{x})\in [0,1]\times \bR^{Nd}$, be the value function of the control problem: 
\begin{equation}\label{reg-sto-control-prob-mf-N}
    \min_{\theta^i\in \Theta; \, i=1,\dots,N}  \mathbb E\left[  U(\mathbf{X}_1)  + \frac{\eps}{2} \sum_{i=1}^N\int_0^1 |\theta^i_t|^2 dt\right],
\end{equation}
subject to the system of controlled SDEs:
\begin{equation}\label{eq:controlled_SDEs_n}
    X_t^i = x_0 + \int_0^t \theta^i_s ds + W_t^i, \quad t\in[0,1], \quad i=1,\dots,N,
\end{equation}
where $x_0\in\bR^d$ and $\{W^i\}_{i=1}^N$ are independent standard $d$-dimensional Brownian motions and the admissible control set $\Theta$ is 
consisting of addmissible controls $\theta$ that are adapted to the filtration generated by $\{W^i\}_{i=1}^N$.
Then this is nothing but a ${dN}$-dimensional version of the control problem \eqref{reg-sto-control-prob}. 
The value function $v_{\eps}^N$ satisfies the HJB equation:
\begin{align}
    \label{eq:particle_hjb}
    \partial_t v_{\eps}^N + \frac{1}{2} \sum_{i=1}^N \Delta_{x_i} v_{\eps}^N - \frac{1}{2\eps} \sum_{i=1}^N |\partial_{x_i} v_{\eps}^N|^2 = 0;
    \quad v_{\eps}^N(1, \mathbf{x}) = U(\mathbf{x}).    
\end{align}
Further, using the Cole-Hopf transformation and the Feynman-Kac formula for such finite-dimensional problems, 
we can linearize this PDE and obtain  an explicit probabilistic representation of the solution as well as of the optimal control, 
i.e.,  
\begin{align}
    v^N_{\eps}(t, \mathbf{x}) &= -\eps \ln\left(\E \left[ \exp\left( -\frac{1}{\eps} U(\mathbf{B}_1) \right) \Bigg| \mathbf{B}_t = \mathbf{x} \right]\right), \label{eq:feynman_kac_N} \\
    \theta_i^*(t, \mathbf{x}) &= \frac{\E \left[ e^{-U(\mathbf{B}_1)/\eps} \cdot \frac{B_1^i - x_i}{1-t} \Bigg| \mathbf{B}_t = \mathbf{x} \right]}{\E \left[ e^{-U(\mathbf{B}_1)/\eps} \Bigg| \mathbf{B}_t = \mathbf{x} \right]}, \quad i=1,\dots,N, 
    \label{optimal-control-particle}
\end{align}
where $\mathbf{B}_s = (B_s^1, \dots, B_s^N)$ represents a system of $N$ independent standard $d$-dimensional Brownian motions which are independent of $\{W^i\}_{i=1}^N$.


\subsubsection{Main result on convergence analysis as $\eps\to 0$ and $N\to \infty$}

Substituting the optimal control \eqref{optimal-control-particle} into the controlled SDEs \eqref{eq:controlled_SDEs_n}, we obtain the SDEs for the optimal state processes $\{X_t^{i,0,x_0;\theta_i^*}\}_{i=1}^N$ 
which at time 1 gives the approximation of the optimization problem \eqref{eq:global_min}.
We establish the convergence of the proposed SCM to the global minimizer of the functional $G$.
Our main theoretical result establishes that the value computed via the $N$-particle approximations converges to the true global minimum of $G$ as   $N \rightarrow \infty$ and $\eps \rightarrow 0$. 

\begin{thm}[Global Convergence of the Method]\label{thm:global_convergence}
Let $\mathcal V_0 =G(\mu^*)= \inf_{\mu} G(\mu)$. Under certain regularity assumptions of $G$, 
the normalized finite-particle value converges to the global minimum, i.e.,
\begin{equation}
    \left| \frac{1}{N} v^N_{\eps}(0, \mathbf{x}^0) - \mathcal V_0 \right| \le \underbrace{\frac{L_{\eps}}{2N}}_{\text{Particle Error}} 
    + \underbrace{C \cdot \eps \ln\left(\frac{1}{\eps}\right)}_{\text{Regularization Error}},
\end{equation}
where $\mathbf{x}^0=(x^0_1,\ldots,x^0_N)=(x_0,\ldots,x_0)\in\bR^{Nd}$,   
$L_{\eps}$ depends on $\eps$ and the regularity of $G$, and the constant $C$ depends on $d$, the H\"older norm and exponent of $G$, and the 2-Wasserstein distance ($\mathcal W_2(\delta_{x_0},\mu^*)$) between $\delta_{x_0}$ and $\mu^*$. 
\end{thm}
The above theorem summarizes the convergence results presented in Theorems \ref{thm:value_convergence} and \ref{thm:regularization_error}.
For clarity of exposition, we assume that all particles start from the same deterministic point $x_0\in\bR^d$.  
More generally, if the initial positions are sampled i.i.d.\ according to an arbitrary law $\mu_0\in\Pcal$, then the total approximation error contains an additional term that accounts for the discrepancy between $\mu_0$ and the empirical initial distribution; see Corollary~\ref{cor:general_measure}.  
Moreover, under the regularity, averaged convexity, and coercivity conditions of Proposition~\ref{prop-well-posedness-mfc}, the constant $L_{\eps}$ in the particle approximation estimate can be chosen uniformly in $\eps$; see Theorem~\ref{thm:value_convergence}.


\subsection{Literature, contributions, and organization of the paper}

\paragraph{Literature.} Optimization problem \eqref{opt-prob} is a central topic in applied mathematics and arises broadly in machine learning,
artificial intelligence, statistics, and operations research; see, e.g., \cite{bottou2018optimization}. 
Given the vast literature on both modeling and algorithms, we restrict attention to a few related directions. 
For general non-convex objectives, computing a global minimizer is generally NP-hard \cite{murty1987some}. 
Accordingly, a wide range of approximation methods have been developed for Euclidean optimization, including first-order schemes 
such as stochastic gradient descent and its variants \cite{robbins1951stochastic,nguyen2021unified,polyak1964some,tripuraneni2018stochastic}, 
as well as derivative-free (zeroth-order) approaches and metaheuristics, including particle swarm optimization \cite{grassi2023mean,kennedy1995particle}, consensus-based optimization \cite{carrillo2018analytical,pinnau2017consensus}, and simulated annealing \cite{kirkpatrick1983optimization,tang2022tail}. 
For a recent overview, we refer to \cite{danilova2022recent,velasco2024literature}.
Optimization over spaces of probability measures (e.g., $\Pcal$) differs fundamentally 
from Euclidean optimization, as the domain is infinite-dimensional and is naturally endowed with a nontrivial geometry, 
typically the Wasserstein geometry. 
Representative approaches include Wasserstein gradient-flow methods 
\cite{ambrosio2005gradient,bunne2022proximal,chizat2018global,hu2021mean,mei2018mean}, 
particle-based methods \cite{chen2023entropic,chizat2018global,nitanda2017stochastic}, 
and conditional-gradient (Frank--Wolfe) methods \cite{jaggi28projection}, among many others.

Stochastic control studies the optimization of controlled dynamical systems evolving under uncertainty, typically modeled by It\^o diffusions. It plays a central role in numerous application areas, including mathematical finance, engineering, reinforcement learning, operations research, and the biological and environmental sciences.  
Two principal analytical frameworks are commonly employed. The first is the Dynamic Programming Principle (DPP), which leads to Hamilton--Jacobi--Bellman (HJB) equations; the second is the stochastic Pontryagin Maximum Principle, which characterizes optimality through systems of forward--backward stochastic differential equations (FBSDEs); 
see, e.g., \cite{fleming2006controlled,YZ99}. An example is the Schrödinger Bridge problem \cite{schrodinger1931uber,chen2016relation}, which can be formulated as a stochastic control problem.
Within this broader field, Mean-Field Control (MFC) and Mean-Field Games (MFG) have emerged as particularly active research directions, motivated by large-population limits and the modeling of interacting multi-agent systems; comprehensive treatments can be found in \cite{cardaliaguet2019master,carmona2018probabilistic}.

\paragraph{Contribution.} In this work, we propose a stochastic control framework for solving the global optimization problem \eqref{opt-prob}, 
by reformulating it as the limit of a class of stochastic controls with regularization terms. 
This reformulation enables the application of dynamic programming principles and the Cole-Hopf transformation to linearize the associated HJB equations. 
The resulting linear PDEs can be solved using probabilistic methods, such as the Feynman-Kac formula, facilitating efficient computation of the value function and optimal control strategies. 
The main contributions of this paper are as follows:
\begin{itemize}
    \item Stochastic control frameworks are introduced for global optimization over $\bR^d$ and $\Pcal$, providing a new perspective on tackling non-convex non-differentiable optimization problems.
    \item For the Euclidean optimization, convergence results are established, demonstrating that values of the control problems converge to the global minimum of the objective functional in the rate of $\eps\ln\left(\frac{1}{\eps}\right)$ 
    as the regularization parameter $\eps$ approaches zero. Although this Euclidean optimization case may be viewed as a special case of the measure optimization, 
    the proofs and techniques are different, the obtained results lay foundation for the measure optimization, and  Euclidean optimization does not necessitate mean-field controls.
    \item The optimization over probability measures is approximated through mean-field control problems which are further approximated via $N$-particle systems (equivalently, $N$-player potential games). We
    establish rigorous convergence results, demonstrating that the value of the controlled finite-particle system converges to the global minimum of the objective functional at the rate of $\frac{L_{\eps}}{N} + C\eps\ln\left(\frac{1}{\eps}\right)$ as the regularization parameter $\eps$ approaches zero
     and the particle number $N$ tends to infinity.
    \item Based on the probabilistic representations of the value functions and optimal controls, we propose the numerical algorithms that are derivative-free due to the utilization of the Bismut-Elworthy-Li formula. Numerical experiments are conducted to validate the theoretical findings and demonstrate the effectiveness of the proposed methods. To reproduce the results, the code is available at\\
     \url{https://github.com/Jinnqiu/Stochastic-Control-Methods-for-Optimization}.
\end{itemize}

\paragraph{Organization of the paper.} The remainder of this paper is structured as follows. 
Section~\ref{section:opt-finite-dim} develops the stochastic control formulation for global optimization in Euclidean spaces, provides a detailed convergence analysis as the regularization parameter $\eps \to 0$, establishes the well-posedness of the associated control problems, and reports numerical experiments. 
Section~\ref{section:optimization-measure} extends the proposed framework to optimization over the Wasserstein space of probability measures, 
presenting the mean-field control formulation, deriving the corresponding convergence results (including the $N$-particle approximation), 
and illustrating the methodology through numerical studies. 
Then the conclusion in Section~\ref{sec:conclusion} summarizes the findings and discusses potential avenues for future research.
Finally, the proofs for Proposition \ref{prop-well-posedness-mfc} and Theorem \ref{thm:value_convergence} are provided in Appendix \ref{app:well-posedness_mfc} and Appendix \ref{app:particle_convergence}, respectively.


\section{Optimization over finite-dimensional spaces: $\mathcal X = \bR^d$}\label{section:opt-finite-dim}
For a matrix $A$, we use $|A|$ to denote its Frobenius/Hilbert-Schmidt norm, while by $|A|_{op}$ we denote its operator norm that is its largest singular value. 
\subsection{A convergence analysis when $\eps\to 0$}
To simplify the arguments and without any loss of generality, throughout this section, we use the following assumption on the objective function $G$ in this subsection.
\begin{ass}\label{assumption-G}
    \begin{enumerate}
        \item For the function $G:\mathbb{R}^d\to\mathbb{R}$, 
        there exist constants $L,p>0$ and $\alpha\in(0,1]$ such that
    \begin{equation}\label{assumption-G-eq}
         |G(x)-G(y)| \leq L(1+|x|^p+|y|^p)|x-y|^{\alpha}, \quad x,y\in\mathbb{R}^d.
    \end{equation}
    \item $E[|x_0|^{2p+2}]<\infty$.  
    \end{enumerate} 
\end{ass}
The local H\"older continuity condition, together with the imposed growth control, accommodates objective functions that may be non-differentiable and non-convex.

Recall the the Brownian bridge $B_t$ connecting $x_0$ and $\xi$ in $\mathbb R^d$:
    \begin{equation}
    B_t=x_0(1-t)+\xi t + (1-t)\int_0^t\frac{1}{1-s}dW_s, \quad t\in[0,t), \quad B_1=\xi.
    \end{equation}
\begin{lem}\label{lem-convergence}
    Let $B$ be a Brownian bridge from $x_0$ to $\xi\in\bR^d $ over the time interval $[0,1]$. Then for each $q\geq 1$, there exists constant $C_q>0$ such that
    \begin{align*}
        \mathbb E\left[ |B_t|^q \right] 
        &\leq C_q \left\{ 
            E\left[|x_0-\xi|^q\right](1-t)^q + d^{q/2}(1-t)^{q/2} t^{q/2} + |\xi|^q
            \right\},\\
        \mathbb E\left[ |B_t-B_1|^q \right] 
        &\leq C_q \left\{ 
            E\left[|x_0-\xi|^q\right](1-t)^q + d^{q/2}(1-t)^{q/2} t^{q/2}  
            \right\},
    \end{align*}
    where the constants $C_q>0$ depend only on $q$.
\end{lem}
\begin{proof}
Applying Burkholder-Davis-Gundy inequality, we have
    \begin{align*}
        \mathbb E\left[
        |B_1-B_t|^q
        \right]
        &= \mathbb E\left[ \left|(x_0-\xi)(1-t) + (1-t)\int_0^t\frac{1}{1-s}dW_s\right|^q \right]
        \\
        &\leq C_q  \left\{ 
            E\left[|x_0-\xi|^q\right](1-t)^q + d^{q/2}(1-t)^q\left| \int_0^t\frac{1}{(1-s)^2}d s \right|^{q/2}
        \right\}
        \\
        &= C_q  \left\{ 
            E\left[|x_0-\xi|^q\right](1-t)^q + d^{q/2}(1-t)^{q/2} t^{q/2} 
            \right\},
    \end{align*}
    and thus, it follows that
    \begin{align*}
        \mathbb E\left[
        |B_t|^q
        \right] 
        &\leq 2^{q-1} \left( \mathbb E[|B_t-\xi|^q] + \mathbb E[|\xi|^q] \right)  \\
        &\leq C_q  \left\{ 
            E\left[|x_0-\xi|^q\right](1-t)^q + d^{q/2}(1-t)^{q/2} t^{q/2} + |\xi|^q
            \right\}.
    \end{align*}
\end{proof}

\begin{thm}\label{thm-convergence}
    Let Assumption \ref{assumption-G} hold. 
   Then, for each $\xi \in \arg\min_{x\in\bR^d}G(x)$, there exists an admissible control $\tilde \theta^{\eps} \in \Theta$ such that the corresponding state process $X^{0,x_0;\tilde\theta^{\eps}}$ satisfies
    \begin{align}
        0 & \leq \mathbb E\left[ G(X_1^{0,x_0;\tilde\theta^{\eps}})  + \frac{\eps}{2} \int_0^T |\tilde    \theta^{\eps}_t|^2 dt\right] - G(\xi) \leq C \, \eps\, \left\{d\ln d +d\ln\left(\frac{1}{\eps}\right)\right\} ,  \label{thm-est-value}
    \end{align}
    where the constant $C>0$ depends only on $p,\alpha,L,  \mathbb E\left[|x_0-\xi|^{2+2p}\right]$, and $|\xi|$. 
\end{thm}

\begin{rmk} 
    Theorem \ref{thm-convergence} indicates that as the regularization parameter $\eps$ approaches zero, 
    the value of the control problem evaluated at the state process $X_1^{0,x_0;\tilde\theta^{\eps}}$ converges to the global minimum of $G$. 
    The convergence rate is at least of order $\eps \ln(1/\eps)$, which highlights the trade-off between the regularization strength and the accuracy of the approximation to the original optimization problem. Here, the function $G$ may have multiple minimizers. While the minimum value $G(\xi)$ does not depend on $\xi$,
    the constant $C$ depends on the choice of the minimizer $\xi$ because we use Brownian bridges in the proof but the choice of $\xi$ is arbitrary within the set of minimizers.
\end{rmk}

\begin{proof}

    Let 
    $B$ be a Brownian bridge from $x_0$ to $\xi$ over the time interval $[0,1]$ as in Lemma \ref{lem-convergence}. Then it satisfies the following stochastic differential equation:
    \begin{equation}
        dB_t = \frac{\xi-B_t}{1-t}dt + dW_t, \quad t\in[0,1), \quad B_0=x_0.
    \end{equation}
    Let $\delta \in (0,1)$ whose value is to be determined. Set $\bar J^{\tilde\theta}=\mathbb E\left[J_{\eps}^{\tilde\theta^{\eps}}(0,x_0)\right]$ with the control process $\tilde\theta^{\eps}_t = \frac{\xi-B_t}{1-t}$ for $t\in[0,1-\delta]$ and $\tilde\theta^{\eps}_t=0$ for $t\in(1-\delta,1]$. 
    Then we have 
    $$
    X^{0,x_0;\tilde{\theta}^{\eps}}_t 
    = B_{t \wedge (1-\delta)} +  W_{t\vee(1-\delta)} - W_{1-\delta} , \quad t\in[0,1],
    $$
    and thus, it holds that
    \begin{align*}
        0 
        \leq \bar J^{\tilde\theta} - G(\xi)
        &= \mathbb E\left[ G(X_1^{0,x_0;\tilde\theta^{\eps}})-G(\xi) + \frac{\eps}{2}\int_0^1 |\tilde\theta^{\eps}_t|^2 dt\right] \\
        &= \mathbb E\left[ G(W_1 - W_{1-\delta} + B_{1-\delta})   -G(B_1)
         + \frac{\eps}{2} \int_0^{1-\delta} \left|\frac{\xi-B_t}{1-t}\right|^2 dt\right] \\
        & =  I_1 + I_2,
    \end{align*}
    where
    \begin{align*}
        I_1 & = \mathbb E\left[ G(W_1 - W_{1-\delta} + B_{1-\delta})   -G(B_1) \right], \\
        I_2 & = \frac{\eps}{2} \mathbb E\left[ \int_0^{1-\delta} \left|\frac{\xi-B_t}{1-t}\right|^2 dt\right].
    \end{align*}

    By the local H\"older continuity in Assumption \ref{assumption-G}, we have
    \begin{align*}
        I_1 & \leq L \mathbb E\left[ (1+|W_1 - W_{1-\delta} + B_{1-\delta}|^p + |B_1|^p) |W_1 - W_{1-\delta} + B_{1-\delta}-B_1|^{\alpha} \right] .
    \end{align*}
    Further, using Jensen's inequality and Lemma \ref{lem-convergence} gives
    \begin{align*} 
        I_1 
       &\leq L C \left(\mathbb E\left[1+|W_1-W_{1-\delta}|^{\frac{2p}{2-\alpha}} + |B_{1-\delta}|^{\frac{2p}{2-\alpha}} + |\xi|^{\frac{2p}{2-\alpha}} \right]\right)^{\frac{2-\alpha}{2}} 
       \left( \mathbb E\left[|W_1 - W_{1-\delta}|^2+|B_{1-\delta}-B_1|^2\right]\right)^{\frac{\alpha}{2}} \\  
       &\leq LC \left(1 + (d\delta)^{\frac{p}{2}} +|\xi|^p + \left(E[|x_0-\xi|^{2+2p}]\right)^{\frac{p}{2+2p}}  
        \right) \left(|d\delta|^{\frac{\alpha}{2}} + |d\delta (1-\delta)|^{\frac{\alpha}{2}}
        +\delta^{\alpha}\left(E[|x_0-\xi|^2]\right)^{\frac{\alpha}{2}}\right),
    \end{align*}
    where $C>0$ is a constant depending only on $p$ and $\alpha$.
    For $I_2$, basic calculation and It\^o isometry yield that
    \begin{align*}
        I_2 
        & = \frac{\eps}{2} \mathbb E\left[ \int_0^{1-\delta} \left|\frac{(\xi - x_0)(1-t) - (1-t)\int_0^t\frac{1}{1-s}dW_s}{1-t}\right|^2 dt\right] \\
        & = \frac{\eps}{2} \mathbb E\left[ |x_0-\xi|^2 (1-\delta) + d\int_0^{1-\delta} \int_0^t \frac{1}{(1-s)^2} ds dt \right] \\
        &\leq \frac{\eps}{2}  \left( E[|x_0-\xi|^2] + d\ln\frac{1}{\delta} \right).
    \end{align*}
    Taking $\delta=\eps^{\frac{2}{\alpha}}d^{-1}$ and combining the estimates of $I_1$ and $I_2$, we obtain
    \begin{align*}
        0 \leq  \bar J^{\tilde\theta^{\eps}} - G(\xi) 
        \leq C \, \eps\, \left\{d\ln d +d\ln\left(\frac{1}{\eps}\right)\right\}   \, ,
    \end{align*}  
    for a constant $C>0$ depending only on $p,\alpha,L,|\xi|$, and  $\mathbb E\left[|x_0-\xi|^{2+2p}\right]$.
   
\end{proof}

\begin{cor}\label{cor-convergence}
    Let $\xi$ be the unique minimizer of $G$.  In Theorem \ref{thm-convergence}, assume further that there exist $\nu, \eta >0$ such that
    \begin{align}
        G(x) - \min_{x\in\bR^d}G(x) \geq \nu |x-\xi|^{\eta}, \quad x\in\bR^d.      \label{bdd-frombelow-growth}  
    \end{align}
    Then we have
    \begin{equation}
        \mathbb E\left[ |X_1^{0,x_0;\tilde\theta^{\eps}}-\xi|^{\eta} \right] \leq C \, \eps\, \left\{d\ln d +d\ln\left(\frac{1}{\eps}\right)\right\}   
    \end{equation}
    where the constant $C>0$ depends only on $p,\alpha,L,\nu,  \mathbb E\left[|x_0-\xi|^{2+2p}\right]$, and $|\xi|$. 
\end{cor}   
\begin{proof}
    The result follows directly from Theorem \ref{thm-convergence} and the additional growth condition \eqref{bdd-frombelow-growth}.
\end{proof} 
The global growth condition \eqref{bdd-frombelow-growth} in Corollary \ref{cor-convergence} may be restrictive in some applications.  The following corollary relaxes this condition to a local one combined with a uniform gap outside a ball.
\begin{cor}\label{cor-convergence-2}
    Let $\xi$ be the unique minimizer of $G$.  In Theorem \ref{thm-convergence}, assume further that there exist constants $\nu, \eta, R, \mathcal E >0$ such that
    \begin{align*}
        G(x) - \min_{x\in\bR^d}G(x) \geq \nu |x-\xi|^{\eta}, \quad \text{for } |x-\xi|\leq R;    \quad 
        G(x)-  G(\xi)\geq \mathcal E, \quad \text{for } |x-\xi|> R. 
    \end{align*}
    Then we have for any $\gamma>0$,
    \begin{equation}
           \mathbb P\left(|\xi-X_1^{0,x_0;\tilde\theta^{\eps}}|>\gamma\right) \leq \frac{C}{\min\{\nu \gamma^{\eta}, \mathcal E\}} \, \eps\, \left\{d\ln d +d\ln\left(\frac{1}{\eps}\right)\right\} ,  
    \end{equation}
    where the constant $C>0$ depends only on $p,\alpha,L, \nu, \mathbb E\left[|x_0-\xi|^{2+2p}\right]$, and $|\xi|$.    
\end{cor}
\begin{proof}
    We just use the estimate in Theorem \ref{thm-convergence}, when $\gamma <R$, we have
    \begin{align*}
        \mathbb P\left(|\xi-X_1^{0,x_0;\tilde\theta^{\eps}}|>\gamma\right) 
        &= \mathbb P\left(|\xi-X_1^{0,x_0;\tilde\theta^{\eps}}|>R\right) + \mathbb P\left(\gamma<|\xi-X_1^{0,x_0;\tilde\theta^{\eps}}|\leq R\right) \\
        &\leq \mathbb P\left(G(X_1^{0,x_0;\tilde\theta^{\eps}}) -G(\xi)>\mathcal E \right)
        + \mathbb P\left(\nu \gamma^{\eta} \leq  G(X_1^{0,x_0;\tilde\theta^{\eps}}) -G(\xi)   \right) \\  
        &\leq \frac{1}{\mathcal E} \mathbb E\left[ G(X_1^{0,x_0;\tilde\theta^{\eps}}) - G(\xi) \right] + \frac{1}{\nu \gamma^{\eta}} \mathbb E\left[ G(X_1^{0,x_0;\tilde\theta^{\eps}}) - G(\xi) \right]\\
        &\leq \frac{C}{\min\{\nu \gamma^{\eta}, \mathcal E\}} \, \eps\, \left\{d\ln d +d\ln\left(\frac{1}{\eps}\right)\right\} , 
    \end{align*}
    where the last inequality follows from Theorem \ref{thm-convergence}.  When $\gamma \geq R$, the same estimate holds obviously.
\end{proof}

In Theorem \ref{thm-convergence} and Corollaries \ref{cor-convergence}--\ref{cor-convergence-2}, the results are still holding if we replace the admissible control $\tilde\theta^{\eps}$ by the optimal control $\theta_{\eps}^*$ if it exists.  
Also note that $\tilde\theta^{\eps}$ is constructed based on the Brownian bridge from $x_0$ to a minimizer $\xi$ of $G$. Although in practice the minimizer $\xi$ is unknown, these results still provide theoretical guarantees for the performance of the regularized stochastic control approach in approximating the global minimum of $G$ as $\eps$ approaches zero.       
In the subsequent section, we will establish the well-posedness of the regularized stochastic control problem \eqref{reg-sto-control-prob} to ensure the existence and characterization of the optimal control $\theta_{\eps}^*$ based on which we propose the numerical approximating schemes.

\subsection{Well-posedness of the stochastic control problems}\label{section:well-posedness}
For simplicity, we assume the following in this subsection.
\begin{ass}\label{assumption-G-global}
    \begin{enumerate}
        \item For the function $G:\mathbb{R}^d\to\mathbb{R}$, there exist constants $L,M>0$ and $\alpha\in(0,1]$ such that
    \begin{equation}\label{eq-G-global}
       \sup_{x\in\bR^d} |G(x)| \leq M
         \quad \text{and}\quad
         |G(x)-G(y)| \leq L |x-y|^{\alpha}, \quad \forall\,x,y\in\mathbb{R}^d.
    \end{equation}
    \item $E[|x_0|^{2}]<\infty$.  
    \end{enumerate} 
\end{ass}

The boundedness assumption does not lose much generality since we can always consider a composition with a bounded function to $G$ (e.g., $\tan^{-1}{G}$) without affecting the minimization problem.
The global H\"older continuity and the square integrability of $x_0$ correspond to the case $p=0$ in Assumption \ref{assumption-G} and they are set to avoid cumbersome arguments, which helps us focus 
on theoretical analysis without loss of much generality.

Let us start with the Feynman-Kac representation \eqref{FK-representation} and rewrite it into an integral form:
\begin{align}
    u(t,x) 
    = \mathbb E\left[ e^{-\frac{1}{\eps}G(W_1-W_t+x)}  \right]
    &= \int_{\bR^d}e^{-\frac{G(x+y)}{\eps}}\frac{e^{-\frac{y^2}{2(1-t)}}}{(2\pi (1-t))^{d/2}} dy.
\end{align}
 Then it is easy to check that $u\in C([0,1]\times \bR^d)\cap C^2([0,1)\times \bR^d)$ satisfying the heat equation \eqref{PDE-u} (see, e.g., \cite{karatzas1998brownian}).  In fact, it is straightforward to see the boundedness: 
 \begin{align}
    e^{-\frac{M}{\eps}} \leq u(t,x) \leq 1, \quad \forall\,(t,x)\in[0,1]\times \bR^d. \label{bounded-u}
 \end{align}
For the gradient, using the convolution representation (change of variable $z=x+y$) we obtain for each $(t,x)\in[0,1)\times \bR^d$,
\begin{align}
    \partial_x u(t,x)   
    &= \partial_x \int_{\bR^d} e^{-\frac{G(z)}{\eps}} \frac{e^{-\frac{|z-x|^2}{2(1-t)}}}{(2\pi(1-t))^{d/2}}\,dz
    = \int_{\bR^d} e^{-\frac{G(z)}{\eps}} \frac{z-x}{1-t}\frac{e^{-\frac{|z-x|^2}{2(1-t)}}}{(2\pi(1-t))^{d/2}}\,dz \nonumber\\
    &= \frac{1}{1-t}\mathbb E\!\left[ e^{-\frac{1}{\eps}G(W_1-W_t+x)}(W_1-W_t)\right], \label{representation-gradient}
\end{align}
where the last identity uses the Gaussian law of $W_1-W_t$. Note $|e^{-x}-e^{-y}|\leq |x-y|$ for all $x,y\geq 0$.
Thus, we have
\begin{align}
    |\partial_x u(t,x)| 
    &= \left|\frac{1}{1-t}\mathbb E\!\left[ \left(e^{-\frac{1}{\eps}G(W_1-W_t+x)}-e^{-\frac{1}{\eps}G(x)}\right)(W_1-W_t)\right]\right| \nonumber\\
    &\leq \frac{1}{1-t}\mathbb E\!\left[ \frac{1}{\eps}|G(W_1-W_t+x)-G(x)| |W_1-W_t|\right] \nonumber \\
    &\leq \frac{L}{\eps(1-t)}\mathbb E\!\left[ |W_1-W_t|^{1+\alpha}\right], \label{bound-gradient-1}
\end{align}
where we used the H\"older continuity of $G$. Consequently, by Jensen's inequality, we obtain
 \begin{align}
     |\partial_x u(t,x)| \le \frac{L(d(1-t))^{\frac{1+\alpha}{2}}}{\eps(1-t)} = \frac{Ld^{\frac{1+\alpha}{2}}}{\eps (1-t)^{\frac{1-\alpha}{2}}}.\label{bounded-1st-derivative}
 \end{align}
 In particular, the feedback control in \eqref{optimal-control-in-terms-of-V} written in terms of $u$,
 \begin{align}
     \theta^*_\eps(t,x)=\frac{\partial_x u(t,x)}{u(t,x)},
 \end{align}
 is well defined and continuous on $[0,1)\times\bR^d$, and satisfies the bound
 \begin{align}
     \sup_{x\in\bR^d}|\theta^*_\eps(t,x)| \le  \frac{Ld^{\frac{1+\alpha}{2}}e^{\frac{M}{\eps}}}{\eps (1-t)^{\frac{1-\alpha}{2}}}.
     \label{bound-optimal-control}
 \end{align}
Analogously, for the second derivatives, we have 
 \begin{align}
     \partial_{xx}^2 u(t,x) 
     &= \int_{\bR^d} e^{-\frac{G(z)}{\eps}} \partial^2_{xx} \left( \frac{e^{-\frac{|z-x|^2}{2(1-t)}}}{(2\pi(1-t))^{d/2}} \right) dz \nonumber \\
     &= \int_{\bR^d} e^{-\frac{G(z)}{\eps}} \left( \frac{(z-x)(z-x)^T}{(1-t)^2} - \frac{I_d}{1-t} \right) \frac{e^{-\frac{|z-x|^2}{2(1-t)}}}{(2\pi(1-t))^{d/2}} dz \nonumber \\
     &= \mathbb{E}\left[ e^{-\frac{1}{\eps}G(W_1-W_t+x)} \left( \frac{(W_1-W_t)(W_1-W_t)^T}{(1-t)^2} - \frac{I_d}{1-t} \right) \right].
 \end{align}
 Analogous to \eqref{bound-gradient-1}, using the H\"older continuity of \( G \), we obtain
 \begin{align}
     |\partial_{xx}^2 u(t,x)| &\leq \frac{L}{\eps(1-t)^2} \mathbb E\left[ |W_1-W_t|^{2+\alpha} \right] + \frac{L}{\eps (1-t)} \mathbb E\left[ |W_1-W_t|^{\alpha} \right] \nonumber \\
     &\leq \frac{CL d^{\frac{2+\alpha}{2}}}{\eps(1-t)^{1-\frac{\alpha}{2}}}, \label{bound-2nd}
 \end{align}
 which together with \eqref{bounded-u} and \eqref{bounded-1st-derivative}  implies that
 \begin{align}
     \sup_{x\in\bR^d}|\partial_x \theta^*_\eps(t,x)| 
     &= \sup_{x\in\bR^d}  \left|\frac{\partial_{xx}^2 u(t,x)}{u(t,x)}-\frac{\partial_x u(t,x) \partial_x ^T u(t,x)}{(u(t,x))^2}\right| \nonumber \\
     &\leq C(L,d,\alpha) \left(\frac{e^{\frac{M}{\eps}}}{\eps(1-t)^{1-\frac{\alpha}{2}}}-\frac{e^{\frac{2M}{\eps}}}{\eps^2 (1-t)^{1-\alpha}} \right).
     \label{bound-optimal-control-gradient}
 \end{align}
Therefore, by the estimates \eqref{bound-optimal-control} and \eqref{bound-optimal-control-gradient}, the standard SDE theory (see, e.g., \cite{karatzas1998brownian}) implies that the SDE \eqref{closed-loop-SDE} admits a unique strong solution on interval $[0,1-\delta]$ for any $\delta \in (0,1)$. In particular, 
due to the boundedness of the feedback control $\theta_{\eps}^*$ in \eqref{bound-optimal-control}, we have 
\begin{align}
\int_0^{1-\delta}\left|\theta_{\eps}^*(t,X^{0,x_0;\theta_{\eps}^*}_t)\right|^2dt 
&\leq 
\left(\frac{L d^{\frac{1+\alpha}{2}} e^{\frac{M}{\eps}}}{\eps}\right)^2
\int_0^{1-\delta} (1-t)^{-(1-\alpha)} \,dt
\nonumber\\
&= \left(\frac{L d^{\frac{1+\alpha}{2}} e^{\frac{M}{\eps}}}{\eps}\right)^2 \frac{1-\delta^{\alpha}}{\alpha}
\nonumber\\
&\le \frac{L^2 d^{1+\alpha} e^{\frac{2M}{\eps}}}{\alpha \eps^2},\quad \text{a.s.} 
\end{align}
Therefore, the unique strong solution may be extended to the whole interval $[0,1]$ and we are ready to present the well-posedness result.

\begin{thm}\label{thm-well-posedness}
Under Assumption \ref{assumption-G-global}, the feedback control $\theta^*_\eps$ defined in \eqref{optimal-control-characterization}  is the optimal control for problem \eqref{reg-sto-control-prob}.
\end{thm}

\begin{proof} 
It is straightforward to check that the function $V_{\eps}(t,x)$ defined through the (inverse) Cole-Hopf transformation in \eqref{value-function-representation} is continuous on $[0,1]\times \mathbb R^d$ and is a classical solution to the HJB equation \eqref{HJB-eqn} over $[0,1)\times \mathbb R^d$. 
Denote by $\bar V_{\eps}(t,x)$ the value function defined in \eqref{value-function}.
Below, we shall verify that $V_{\eps}$ in \eqref{value-function-representation} coincides with $\bar V_{\eps}$ and that  $\theta^*_\eps$ defined in \eqref{optimal-control-characterization}  is the optimal control.
For each $(t,x)\in[0,1)\times \bR^d$ and control $\theta\in \Theta$, applying It\^o formula gives that for all $\tau\in[t,1)$,
\begin{align}
    V_{\eps}(t,x) 
    &= \mathbb E\left[V_{\eps}(\tau, X_{\tau}^{t,x;\theta}) -\int_t^{\tau}\left(\partial_s V_{\eps}
    +\frac{1}{2}\Delta V_{\eps} +\theta \cdot \partial_x V_{\eps} \right)(s,X_s^{t,x;\theta})ds
    -\int_t^{\tau}\partial_x V_{\eps}(s,X_s^{t,x;\theta}) dW_s \right]\nonumber\\
    & = \mathbb E\left[V_{\eps}(\tau, X_{\tau}^{t,x;\theta}) -\int_t^{\tau}\left(\partial_s V_{\eps}
    +\frac{1}{2}\Delta V_{\eps} +\theta \cdot \partial_x V_{\eps}  \right)(s,X_s^{t,x;\theta})ds
     \right],\label{verification-ito}
\end{align}
where we have used the boundedness of $\partial_x V_{\eps}$ by \eqref{optimal-control-characterization} and \eqref{bound-optimal-control} for the mean-zero stochastic integral.
Recalling the HJB equation \eqref{HJB-eqn}, we notice that
\begin{align*}
    \partial_s V_{\eps}
    +\frac{1}{2}\Delta V_{\eps} +\theta \cdot \partial_x V_{\eps}  
    & = \partial_s V_{\eps}
    +\frac{1}{2}\Delta V_{\eps}  -\frac{|\partial_x V_{\eps}|^2}{2\eps} + \frac{\eps}{2}\left|\theta+\frac{\partial_x V_{\eps}}{\eps}\right|^2
    - \frac{\eps}{2} |\theta|^2\\
    &= \frac{\eps}{2}\left|\theta+\frac{\partial_x V_{\eps}}{\eps}\right|^2
    - \frac{\eps}{2} |\theta|^2\\
    &\leq 
    - \frac{\eps}{2} |\theta|^2.
\end{align*}
Thus, continuing \eqref{verification-ito} yields that
\begin{align}
    V_{\eps}(t,x) 
    &\leq \mathbb E\left[V_{\eps}(\tau, X_{\tau}^{t,x;\theta}) +\frac{\eps}{2}\int_t^{\tau} |\theta_s|^2ds \right]. \label{ineq-ito}
\end{align}
Letting $\tau$ tend to $1$, by the dominated convergence theorem and the continuity of $V_{\eps}$, we obtain
\begin{align}
    V_{\eps}(t,x) 
    &\leq \mathbb E\left[V_{\eps}(1, X_{1}^{t,x;\theta}) +\frac{\eps}{2}\int_t^{1} |\theta_s|^2ds \right]
    \nonumber\\
    & =\mathbb E\left[G( X_{1}^{t,x;\theta}) +\frac{\eps}{2}\int_t^{1} |\theta_s|^2ds \right]\nonumber\\
    &= J^{\theta}_{\eps}(t,x),
\end{align}
which by the arbitrariness of $\theta$ implies that 
\begin{equation}\label{verification-leq}
    V_{\eps}(t,x)\leq \bar V_{\eps}(t,x).
\end{equation}

On the other hand, applying It\^o formula to $V_{\eps}(s,X_{s}^{t,x;\theta_{\eps}^*})$ with the feedback control $\theta_{\eps}^*$ defined in \eqref{optimal-control-characterization}, we may have the equality in \eqref{ineq-ito}
and letting $\tau$ tend to $1$ yields that 
$$
V_{\eps}(t,x)=J^{\theta_{\eps}^*}_{\eps}(t,x) \geq \bar V_{\eps}(t,x),
$$
which together with \eqref{verification-leq} implies that $V_{\eps}(t,x)=\bar V_{\eps}(t,x)$ for all $(t,x)\in [0,1]\times \bR^d$ and that $\theta_{\eps}^*$ is the optimal feedback control.
\end{proof}

Under Assumption \ref{assumption-G-global}, we may have a finer convergence analysis as in Theorem \ref{thm-convergence}. Indeed, in view of the fact that for each $\theta\in\Theta$,
\begin{align*}
    0\leq \mathbb E\left[G(X_1^{0,x_0;\theta_{\eps}^*})\right] -\min_{x\in\bR^d} G(x)
    \leq \mathbb E\left[J^{\theta_{\eps}^*}_{\eps}(0,x_0)\right] - \min_{x\in\bR^d} G(x) 
    \leq \mathbb E \left[J_{\eps}^{\theta}(0,x_0)\right]- \min_{x\in\bR^d} G(x), 
\end{align*}
we may repeat the proof of Theorem \ref{thm-convergence} to obtain the following result.

\begin{thm}\label{thm-convergence-2}
Under Assumption \ref{assumption-G-global}, for each $\xi \in \arg\min_{x\in\bR^d}G(x)$, it holds that
    \begin{equation}
        0 \leq \mathbb E\left[ G(X_1^{0,x_0;\theta_{\eps}^*})  + \frac{\eps}{2} \int_0^T |\theta_{\eps}^*(t,X_t^{0,x_0;\theta_{\eps}^*})|^2 dt\right] - G(\xi) 
        \leq C \, \eps\, \left\{d\ln d +d\ln\left(\frac{1}{\eps}\right)\right\} ,  
    \end{equation}
    where the constant $C>0$ depends only on $\alpha,L$, and $  \mathbb E\left[|x_0-\xi|^2\right]$.
\end{thm}
The proof is so analogous to that of Theorem \ref{thm-convergence} that it is omitted. Nevertheless, we note that for the constant $C$ herein, the dependence on $|\xi|$ is waived due to the global H\"older continuity \eqref{eq-G-global}.

\subsection{Numerical approximations  with examples}\label{subsection:numerics_finite_dim}
In view of Theorems \ref{thm-convergence} and \ref{thm-convergence-2}, 
we see that the regularized optimal control problem \eqref{reg-sto-control-prob} 
can approximate the optimization problem \eqref{opt-prob} with an error rate of order $\eps\, \ln\left(\frac{1}{\eps}\right)$.
This result provides a theoretical guarantee for the accuracy of numerical approximations based on the regularized stochastic control approach. 
Then it remains to solve the regularized optimal control problem \eqref{reg-sto-control-prob} for which 
there are two streams of methods Pontryagin's maximum principle and Bellman's dynamic programming principles (refer to, e.g., \cite{YZ99}). The former gives the necessary 
conditions in terms of forward-backward SDEs for the optimal control and state, while the latter resorts to the nonlinear HJB equation whose sufficiently regular solution yields 
both the value function and the optimal (feedback) control. From the numerical point of view, numerical approaches for approximating FBSDEs include 
time discretization schemes such as the Euler method, regression-based Monte Carlo algorithms, and more recently, 
deep learning-based methods that leverage neural networks for high-dimensional FBSDEs, while the existing numerical methods  for HJB equations include
finite difference methods, finite elements methods, policy/value function iterations, and recently the deep learning-based BSDE methods; 
refer to \cite{forsyth2007numerical,han2018solving,han2025brief,hure2020deep,kharroubi2013numerical,mceneaney2007curse,witte2011penalty} among many others. 
Some of these methods are curse-of-dimensionality free. In particular, the deep BSDE method introduced in \cite{han2018solving} is exactly applied to the HJB equation like \eqref{HJB-eqn} and it is shown that 
the curse of dimensionality is overcome.

In this work, we investigate the method based on simulating the SDE for the optimal state \eqref{closed-loop-SDE}. 
Rewrite the SDE \eqref{closed-loop-SDE} for the optimal state process: for $t\in[0,1]$,
\begin{equation*}
    X_t^{0,x_0;\theta_{\eps}^*} = x_0 
    + \int_0^t \frac{1}{1-s}\left( \beta(s,X_s^{0,x_0;\theta_{\eps}^*}) - X_s^{0,x_0;\theta_{\eps}^*} \right) ds + W_t,
\end{equation*} 
where 
$$\beta(s,x) =  \frac{\mathbb E\left[ e^{-\frac{1}{\eps}G(W_1-W_s+x)}(W_1 - W_s+x)  \right]}{ \mathbb E\left[ e^{-\frac{1}{\eps}G(W_1-W_s+x)}   \right]} .$$
This SDE does not have explicit solution and we may use the Euler-Maruyama scheme to approximate.
Given a time step size $\Delta t = 1/M$ for some integer $M>0$, we define the time grid $t_k = k \Delta t$ for $k=0,1,\ldots,M$. The Euler-Maruyama scheme for \eqref{closed-loop-SDE} is given by
\begin{align}\label{Euler-scheme-1}
    X_0=x_0,
    \quad
    X_{t_{k+1}}  
    = X_{t_k} 
    +  \frac{1}{1-t_k}\left( \beta(t_k,X_{t_k}) - X_{t_k}    \right) 
    \Delta t 
    + \sqrt{\Delta t} \, Z_{k+1}, 
\end{align}
where $\{ Z_k \}_{k=1}^M$ are independent standard normal random variables in \(\mathbb{R}^d\). 
For each time step,  we need to evaluate $\beta(t_k, X_{t_k})$, which by \eqref{optimal-control-characterization} may be approximately computed as  
\begin{align}
    \beta\!\left(t_k,X_{t_k}\right)
    &= \frac{\mathbb E\!\left[ \exp\!\left(-\frac{1}{\eps}G( \sqrt{1-t_k} \xi+x)\right)\,( \sqrt{1-t_k} \xi+x) \right]}
    {\mathbb E\!\left[ \exp\!\left(-\frac{1}{\eps}G(\sqrt{1-t_k} \xi+x)\right) \right]}
    \bigg|_{x=X_{t_k}}\nonumber \\
    &\approx
    \sum_{i=1}^S   \frac{\exp\!\left(-\frac{1}{\eps}G(\sqrt{1-t_k} \xi_i+X_{t_k})\right) ( \sqrt{1-t_k} \xi_i+X_{t_k})}
    {
    \sum_{i=1}^S \exp\!\left(-\frac{1}{\eps}G(\sqrt{1-t_k} \xi_i+X_{t_k })\right)} 
    \, , \label{beta-finite-dim}
\end{align}
where $\xi$ is a standard normal random variable in \(\mathbb{R}^d\) independent of other random variables, $\{\xi_i\}_{i=1}^S$ are independent samples of $\xi$, 
and $S$ is the sample size.   We may simulate
multiple ($N$) independent trajectories of the SDE \eqref{closed-loop-SDE} and take  the mean  of the terminal states as the approximation of the global minimizer of $G$.
To enhance the exploration of the state space, we may iterate the above procedure multiple times, 
each time initializing the particles with a convex combination of the terminal states from the previous iteration and their mean. 
Two considerations motivate the coupling step in which the next initialization is taken as a convex combination of the terminal particle locations and their empirical mean.
First, by the law of large numbers, as $N\to\infty$ the empirical mean of terminal states converges to the mean of the optimally controlled state, i.e.,
$\frac1N\sum_{i=1}^N X^{(i)}_1 \to \E[X^{0,x_0;\theta^*_\eps}_1]$,
so that the empirical mean provides a statistically stable proxy for the optimizer recovered by \eqref{closed-loop-SDE}.
Second, the empirical mean aggregates the terminal location information across particles and thereby induces an explicit coupling among trajectories, which can improve exploration and reduce sensitivity to individual sample paths.

More precisely, writing $m_s := \E\!\left[X_s^{0,x_0;\theta_{\eps}^*}\right]$ and taking expectations in \eqref{closed-loop-SDE} yields the deterministic evolution
\begin{align*}
    \frac{d}{ds} \E\left[X_s^{0,x_0;\theta_{\eps}^*}\right] 
    &= \frac{1}{1-s }\left( \E\left[\beta(s,X_s^{0,x_0;\theta_{\eps}^*})\right] - \E\left[X_s^{0,x_0;\theta_{\eps}^*}\right] \right) \\
    &= \frac{1}{1-s }\left(  \overline X^{\eps}_s - \E\left[X_s^{0,x_0;\theta_{\eps}^*}\right] \right), 
\end{align*}
where 
$$ \overline X^{\eps}_s = \frac{\mathbb E\left[ e^{-\frac{1}{\eps}G(W_1-W_s+X_s^{0,x_0;\theta_{\eps}^*})}(W_1 - W_s+X_s^{0,x_0;\theta_{\eps}^*})  \right]}{ \mathbb E\left[ e^{-\frac{1}{\eps}G(W_1-W_s+X_s^{0,x_0;\theta_{\eps}^*})}   \right]},$$
which incorporates the position information of other particles, and by Laplacian principle concentrates around the global minimizer of $G$ as $\eps$ approaches zero.
Therefore, the iteration step combining terminal states and their mean can be viewed as a coupling mechanism among particles to enhance exploration and accelerate convergence to the global minimizer.

The pseudocode is presented in Algorithm \ref{alg:particle_method}, which we refer to as the Stochastic Control Method for Global Optimization on $\bR^d$. 
The output $\bar{x}^*$ of Algorithm \ref{alg:particle_method} serves as the approximation of the global minimizer of $G$,  and its accuracy is mainly guaranteed by Theorems \ref{thm-convergence} and \ref{thm-convergence-2}.

\begin{algorithm}[H]
\caption{Stochastic Control Method for Global Optimization on $\bR^d$}
\label{alg:particle_method}
\begin{algorithmic}[1]
\State \textbf{Parameters:} Dimension $d$, Particles $N$, Time steps $M$, Horizon $T$, Regularization $\eps$, MC samples $S$, Iterations $L$, Coupling $\lambda$.
\State \textbf{Initialize:} $X_0^{(i)} \leftarrow x_{\text{start}}$ for $i=1,\dots,N$, \State $\Delta t \leftarrow T/(M-1)$
\For{iteration $l = 1$ to $L$}
    \State Sample $Z_1, \dots, Z_S \sim \mathcal{N}(0, I_d)$
    \For{step $m = 0$ to $M-1$}
        \State $t \leftarrow m \Delta t$, $\tau \leftarrow T - t$
        \For{particle $i = 1$ to $N$}
             \State Compute weights $w_j \leftarrow \exp\left(-\frac{G(X_t^{(i)} + \sqrt{\tau} Z_j) - \min_k G(X_t^{(i)} + \sqrt{\tau} Z_k)}{\eps}\right)$
            \State Estimate expectation $E \leftarrow \frac{\sum_{j=1}^K w_j (X_t^{(i)} + \sqrt{\tau} Z_j)}{\sum_{j=1}^K w_j}$
            \State Compute drift $\hat{\theta} \leftarrow \frac{E - X_t^{(i)}}{\tau}$
            \State Step: $X_{t+\Delta t}^{(i)} \leftarrow X_t^{(i)} + \hat{\theta} \Delta t + \sqrt{\Delta t} \xi^{(i)}$, with $\xi^{(i)} \sim \mathcal{N}(0, I_d)$
        \EndFor
    \EndFor
    \State \textbf{Update Start:} $X_0^{(i)} \leftarrow \lambda \left(\frac{1}{N}\sum X_T^{(j)}\right) + (1-\lambda) X_T^{(i)}$  if $l < L$
\EndFor
\State \textbf{Output:} $\bar{x}^* \approx \frac{1}{N} \sum_{i=1}^N X_T^{(i)}$
\end{algorithmic}
\end{algorithm}

To illustrate the proposed method, we present two standard benchmark problems from global optimization.

\begin{exam}[Xin-She Yang 4 Function]
The Xin-She Yang 4 function is a \textit{non-convex, non-differentiable} function defined as:
\begin{equation*}
g(x) = \left( \sum_{i=1}^d \sin^2(x_i) - \exp\left(-\sum_{i=1}^d x_i^2\right) \right) \exp\left(-\sum_{i=1}^d \sin^2\sqrt{|x_i|}\right).
\end{equation*}
The global minimum is located at the origin $x^*=(0,\dots,0)$ with value $g(x^*)=-1$.

The simulation parameters are set as follows: dimension $d=1$, number of particles $N=20$, time horizon $T=1.0$, and number of time steps $M=4001$.
The regularization parameter is set to $\eps=10^{-300}$, and we use $S=800$ Monte Carlo samples for the drift estimation.
The initial state for all particles is set to $x_0 = 2$, and the mean-field coupling weight (not used) is $\lambda=0.5$. 
The algorithm is run for 1 iteration. The final mean vector: $\bar{X}_T \approx 0.00167204$ 
(the right end point of the black line in Figure \ref{fig:yang_trajectories}).
\begin{figure}[htbp]
    \centering
    \includegraphics[width=0.8\textwidth]{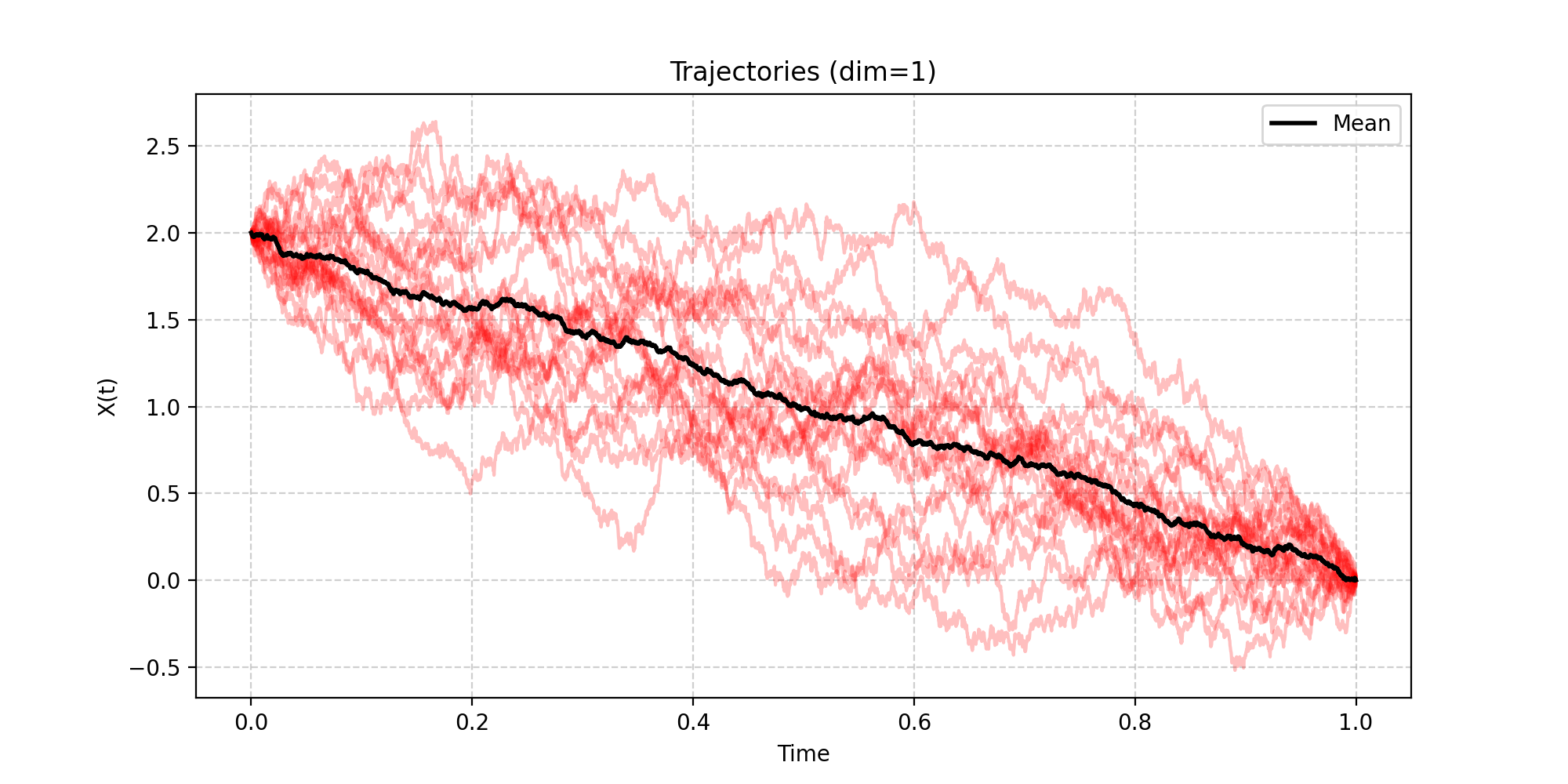}
    \caption{Trajectories of 20 particles for the Xin-She Yang 4 function optimization. The particles converge to the global minimum at the origin.}
    \label{fig:yang_trajectories}
\end{figure}

To corroborate the convergence analysis in Theorem \ref{thm-convergence}, we use 1000 particles, set $x_0=1$, and take 21 values of $\eps$ evenly placed between $0.0005$ and $0.12$, and plot the 
the corresponding value error (Monte Carlo estimate of $V_{\eps}(0,x_0)$ minus $G(0)$) against $\eps$ and $-\eps\ln(\eps)$, 
respectively, in Figure \ref{fig:yang_convergence}.
\begin{figure}[htbp]
    \centering
    \includegraphics[width=0.95\textwidth]{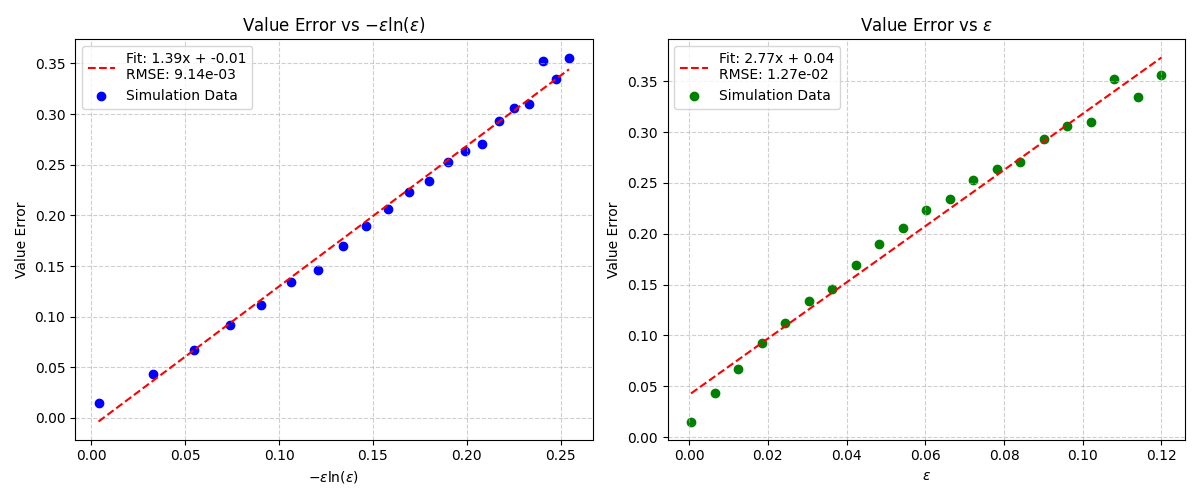}
    \caption{Convergence of the Stochastic Control Method for the Xin-She Yang 4 function optimization. The left panel shows the value error against $-\eps\ln(\eps)$, while the right panel shows the value error against $\eps$. 
    The linear trend in the left panel provides a better fit, as indicated by the lower RMSE, corroborating the theoretical convergence rate established in Theorem \ref{thm-convergence-2}.}
    \label{fig:yang_convergence}    
\end{figure}

\end{exam}

The following example illustrates the applicability of the proposed method to high-dimensional, non-convex optimization problems.
\begin{exam}[20D Ackley Function]
The Ackley function is a typical benchmark for testing optimization algorithms due to its large number of local minima:
\begin{equation*}
g(x) = -20 \exp\left(-0.2 \sqrt{\frac{1}{d} \sum_{i=1}^d x_i^2}\right) - \exp\left(\frac{1}{d} \sum_{i=1}^d \cos(2\pi x_i)\right) + 20 + e.
\end{equation*}
The global minimum is at $x^*=(0,\dots,0)$ with $g(x^*)=0$.
In this experiment, we consider a $20$-dimensional instance of the Ackley function and set the number of particles to $N=1000$. 
The simulation horizon is $T=1.0$, discretized with $M=2001$ time steps, and the mean-field coupling weight is $\lambda=0.75$. 
We choose the regularization parameter $\eps=10^{-300}$ and use $S=1000$ Monte Carlo samples at each time step to estimate the drift term.
All particles are initialized at $x_0=(5,\dots,5)^\top$, and the algorithm is executed for $10$ outer iterations.
At termination, the empirical mean of the particle system is
\begin{align*}
    \bar{X}_T \approx\, & [ 0.00677943,\, -0.02596362,\,  0.02645172,\, -0.01391329,\,  0.00043321,\, -0.00798279,\\
    &\quad  -0.02765826,\, -0.00654247,\, -0.00301501,\,  0.00804032,\,  0.01024422,\,   0.01322995,\\
 &\quad -0.03054515,\,  0.01779264,\, -0.00474688,\, -0.00306365,\, -0.02286298,\,  -0.00381099,\\
 &\quad 0.00962246,\,  0.00973099],
\end{align*}
which is close to the global minimizer at the origin.
\end{exam}

It is worth noting that as our primary emphasis is on the theoretical development and analysis of the proposed SCM for global optimization,
the numerical experiments above are presented solely for illustrative purposes, 
and the reported parameter choices are not optimized 
for accuracy or computational efficiency. In practical implementations, in addition to key parameters such as the regularization level $\eps$, 
the particle number $N$, the Monte Carlo sample size $S$, the number of outer iterations $L$, 
the coupling weight $\lambda$, and the time discretization level $M$, one may further tune the time horizon $T$, 
employ more sophisticated Monte Carlo estimators or variance-reduction techniques, 
and introduce suitable scalings of the driving Wiener process $W$. 
A comprehensive numerical investigation, including systematic parameter calibration and detailed comparisons with established global optimization algorithms, 
is left for future work.


\section{Optimization over probability measures: $\mathcal X= \Pcal$}\label{section:optimization-measure}

This section is devoted to the SCM introduced in Section~\ref{section:opt-prob-measure} for global optimization over the space of probability measures. 
We provide a convergence analysis of the proposed particle approximation scheme, including quantitative bounds for both the finite-particle approximation error and the regularization error.



\subsection{Finite Particle Approximation}
The first source of error comes from approximating the continuous mean-field control problem with controlled finite number of particles. 

First, we recall Lions' derivative, which is the key tool for analyzing functions defined on the space of probability measures; see \cite{cardaliaguet2019master,carmona2018probabilistic} for more details.
Let $\Pcal$ be the space of probability measures on $\bR^d$ with finite second moments, equipped with the $2$-Wasserstein metric $\mathcal W_2(\cdot,\cdot)$.
A function $U:\Pcal\to \bR$ is said to be \textit{(Lions) differentiable} at $\mu\in \Pcal$ 
if there exists a random variable $\xi$ defined on some probability space $(\Omega,\mathcal F,\mathbb P)$ 
with law $\mathcal{L}(\xi)=\mu$ such that the mapping
\begin{align*}
    L^2(\Omega,\mathcal F,\mathbb P;\bR^d)\ni X \mapsto U(\mathcal{L}(X)) \in \bR
\end{align*}
is Fr\'echet differentiable at $X=\xi$. In this case, there exists a unique function $\Lions U(\mu):\bR^d\to \bR^d$ such that
\begin{align*}
     U(\cL(\xi+Y))-U(\cL(\xi)) 
     = \mathbb E\left[ \Lions U(\mu)(\xi) \cdot Y \right]
     +o\left(\sqrt{\E[|Y|^2]}\right), \quad \forall\, Y\in L^2(\Omega,\mathcal F,\mathbb P;\bR^d).
\end{align*}
The function $\Lions U(\mu)$ is called the \textit{(Lions) derivative} of $U$ at $\mu$. 
If the mapping $(\mu,x)\mapsto \Lions U(\mu)(x)$ is continuous on $\Pcal\times \bR^d$, then $U$ is said to be \textit{continuously (Lions) differentiable}.
Higher-order (Lions) derivatives can be defined similarly; in particular, the second-order (Lions) derivative $\partial^2_{\mu\mu} U(\mu)(x)(y)$ is a matrix in $\bR^{d\times d}$ for each $(\mu,x,y)\in \Pcal\times \bR^d\times \bR^d$.

To focus on the study of particle approximations, we assume the following smoothness condition on the solution to the master equation \eqref{Master-eq}.

\begin{ass}\label{ass:smoothness}
Let $G$ be so smooth that for each $\eps>0$ the master equation \eqref{Master-eq} admits a classical solution $\mathcal V_\eps$ satisfying:
\begin{enumerate}
    \item [(a)] the functions   $\partial_t \mathcal V_{\eps}$,$\partial_{\mu}\mathcal V_\eps$, $\partial_x\partial_{\mu}\mathcal V_\eps$,  $ \partial^2_{\mu\mu} \mathcal{V}_{\eps}$ are continuous and $\partial_x\partial_{\mu}\mathcal V_\eps$ and $\partial^2_{\mu\mu} \mathcal{V}_{\eps}$ are bounded,
    \item [(b)] it holds that $\left|\int_{\mathbb R^d}
\text{Tr}\left(\partial^2_{\mu\mu} \mathcal{V}_{\eps}(t, \mu)(y, y)\right) \mu(dy) \right| \leq L_{\eps}<\infty$,   for all $(t,\mu)\in [0,1]\times \Pcal$. 
\end{enumerate}
\end{ass}

From Assumption \ref{ass:smoothness}, we may derive the uniqueness of the classical solution as well as the characterization of the optimal feedback control.
\begin{lem}\label{lem:uniqueness}
Under Assumption \ref{ass:smoothness}, for each $\eps>0$, the classical solution $\mathcal V_\eps$ to 
the master equation \eqref{Master-eq} is the value function defined in \eqref{value-function-MF} and thus is unique.
\end{lem}
\begin{proof}[Sketch of the proof.]
The uniqueness of the classical solution to the master equation \eqref{Master-eq} follows from the verification argument. 
Indeed, applying the It\^o formula to any classical solution $\tilde{\mathcal{V}}_{\eps}$ 
of the master equation  \eqref{Master-eq} yields the fundamental identity: 
for any control $\theta \in \Theta$,
\begin{align*}
    J^{\theta}_{\eps}(t, \mu) 
     =  \tilde{\mathcal{V}}_{\eps}(t, \mu) + \mathbb{E} \left[ \int_t^T \frac{1}{2\eps} \left| \eps\, \theta_s + \Lions \tilde{\mathcal{V}}_{\eps}(s, \mathcal{L}(X^{t,\xi;\theta}_s))(X^{t,\xi;\theta}_s) \right|^2 ds \right].   
\end{align*}
Thus 
$$\tilde{\mathcal{V}}_{\eps}(t, \mu) \le J^{\theta}_{\eps}(t, \mu), \quad \forall\, \theta \in \Theta,$$
and the equality holds when choosing the optimal feedback control $\theta^*_s=-\frac{\Lions \tilde{\mathcal{V}}_{\eps}(s, \mathcal{L}(X^{t,\xi;\theta^*}_s))(X^{t,\xi;\theta^*}_s)}{\eps}$.
Hence, $\tilde{\mathcal{V}}_{\eps} = \mathcal{V}_{\eps}$.   
\end{proof}

While Assumption \ref{ass:smoothness} is stronger than what is required to guarantee the convergence of the normalized particle value function to the mean-field value function
 (see, e.g., \cite{bayraktar2025viscosity,carmona2018probabilistic,cheung2025viscosity,cosso2024master,lacker2017limit} for convergence results under weaker notions of solution), 
 additional regularity is typically indispensable for establishing quantitative rates of convergence; see \cite{cardaliaguet2019master,carmona2018probabilistic,chassagneux2014probabilistic,daudin2024optimal,germain2022rate}.
Accordingly, in line with the approach of \cite{germain2022rate}, we impose Assumption \ref{ass:smoothness} and work under the existence of a classical solution to the master equation 
to streamline the derivation of an $\mathcal{O}(1/N)$ error estimate.
For results on the existence of classical solutions to master equations under suitable structural and smoothness conditions, 
we refer to \cite{bayraktar2025convergence,buckdahn2017mean,cardaliaguet2019master,carmona2018probabilistic,chassagneux2014probabilistic} among many others. 
For completeness, we also present below a well-posedness result under classical $C^2$-type regularity and convexity assumptions on $G$. 
A key point we want to show is that the corresponding a priori estimates for the Lions derivatives of the value function can be taken \emph{uniformly} in the regularization parameter $\eps$.
\begin{prop}[Well-posedness of the master equation]\label{prop-well-posedness-mfc}
Suppose 
\begin{enumerate}
    \item [(a)] (Regularity): $G$ has Lipschitz continuous derivatives $\partial_\mu G(\mu)(x)$, $\partial^2_{\mu\mu} G(\mu)(x)(y)$, $\partial_x\partial_{\mu}G(\mu)(x)$, $\partial^2_{xx}\partial_{\mu}G(\mu)(x)$,
    $\partial_x\partial^2_{\mu\mu}G(\mu)(x)(y)$, and $\partial_y\partial_{\mu\mu}^2G(\mu)(x)(y)$ and there exists a constant $K_G>0$ such that 
        $|\partial^2_{\mu\mu}G(\mu)(x)(y)| + |\partial_x\partial_{\mu}G(\mu)(x) |
        \le K_G$ for all $\mu \in \Pcal$ and $x,y \in \bR^d$;
    \item [(b)] (Averaged convexity): For each $\mu\in\Pcal$ and every pair $(X,\eta)$ of $\bR^d$-valued square integrable random variables with $\cL(X)=\mu$,
        \begin{equation*}
            \E\left[
            \frac{1_{|\eta|>0}}{|\eta|}\langle \eta,\partial_x\partial_\mu G(\mu)(X)\eta\rangle
            + \frac{1_{|\eta|>0}}{|\eta|}
            \Big\langle \eta,\widetilde{\E}\big[\partial^2_{\mu\mu}G(\mu)(X)(\widetilde X)\widetilde\eta\big]\Big\rangle \right]
            \ge 0
        \end{equation*}
        where $(\widetilde X,\widetilde\eta)$ is an independent copy of $(X,\eta)$.
    \item [(c)] (Coercivity) There exists $\kappa>0$ such that 
         \[
             \partial_x \partial_\mu G(\mu)(x) \geq \kappa I_d, \quad \text{for all } (x,\mu) \in \mathbb{R}^d\times \Pcal. \]
\end{enumerate}
Then the value function $\mathcal{V}_{\eps}$ of the mean-field control problem \eqref{value-function-MF} 
is the unique classical solution to the master equation \eqref{Master-eq} (as in Assumption \ref{ass:smoothness}). 
Moreover, $\mathcal{V}_{\eps}$ satisfies the 
following $\eps$-uniform estimates: 
for all $(t,\mu,x)\in [0,T]\times \Pcal \times \bR^d$
\begin{align}
    & |\partial_x\partial_{\mu} \mathcal{V}_{\eps}(t, \mu) (x)|_{op} \leq  K_G , 
    \nonumber\\
    &|\partial^2_{\mu\mu} \mathcal{V}_{\eps}(t, \mu)(x)(y)|_{op} \leq \frac{(K_G)^2}{\kappa} + K_G, \quad \mu\text{-a.e. } y\in\bR^d. \label{boundedness-Lions-derivative}
\end{align}
\end{prop}

The identification of the value function $\mathcal V_{\eps}$ as a classical solution to the master equation \eqref{Master-eq} 
follows from the stochastic Pontryagin maximum principle in \cite{carmona2018probabilistic} and the classical solution theory of master equations in \cite{chassagneux2014probabilistic};
 in particular, this justification does not require the assumption (c) (Coercivity) in Proposition \ref{prop-well-posedness-mfc}.
In contrast, the coercivity condition plays an essential role in our argument for deriving the $\eps$-uniform bounds in \eqref{boundedness-Lions-derivative}. 
Our analysis is devoted to establishing these $\eps$-uniform estimates, which, to the best of our knowledge, are not addressed explicitly in the existing literature. 
For the sake of readability, the technical proof is deferred to Appendix~\ref{app:well-posedness_mfc}.

\begin{thm}[Particle Approximation Error]
\label{thm:value_convergence}
Let Assumption \ref{ass:smoothness} hold and let $v^N_\eps(t, \mathbf{x})$ be the value function of the $N$-particle control problem satisfying  the HJB equation \eqref{eq:particle_hjb}.
Then, for any configuration $\mathbf{x} \in \bR^{Nd}$, the per-particle value function converges with rate $O(1/N)$, i.e.,
\begin{equation}
    \left| \frac{1}{N} v^N_{\eps}(t, \mathbf{x}) - \mathcal{V}_\eps(t, \mu^N_{\mathbf{x}}) \right| \le \frac{L_{\eps}}{2   N}, \quad \forall\, t \in [0,1],
\end{equation}
where $\mu^N_{\mathbf{x}}$ is the empirical measure of the configuration. Moreover, the constant $L_{\eps}$ is independent of $\eps$ under assumptions in Proposition \ref{prop-well-posedness-mfc}.
\end{thm}

A proof, inspired by the projection argument developed in \cite{germain2022rate} 
and based on combining the projection of the master-equation solution onto the empirical manifold with a comparison principle for the finite-dimensional HJB equation, is given in Appendix~\ref{app:particle_convergence}.

To extend the result to a general probability measure $\mu \in \Pcal$, 
one may rely on the Lipschitz continuity of the value function $\mathcal{V}_\eps$; note that we do not need such an extension in Theorem \ref{thm:global_convergence} 
as we choose all particles starting from a fixed point. In fact, the value function $\mathcal{V}_\eps$ is Lipschitz continuous with respect to the 2-Wasserstein distance under the Lipschitz assumptions on $G$.
 For any $\mu_1,\mu_2\in \Pcal$, by the definition of $\mathcal{V}_\eps$, we have
\begin{align*}
    |\mathcal{V}_\eps(t,\mu_1) - \mathcal{V}_\eps(t,\mu_2)|
    =\Big|\inf_{\theta \in \Theta} J^{\theta}_{\eps}(t,\mu_1) - \inf_{\theta \in \Theta} J^{\theta}_{\eps}(t,\mu_2) \Big| 
    &\le \sup_{\theta \in \Theta} \left| G(\mathcal{L}(X_1^{t,\mu_1;\theta})) - G(\mathcal{L}(X_1^{t,\mu_2;\theta})) \right| \\
    &\le L_G \sup_{\theta \in \Theta} \mathcal W_2(\mathcal{L}(X_1^{t,\mu_1;\theta}), \mathcal{L}(X_1^{t,\mu_2;\theta})) \\
    &\le L_G \mathcal W_2(\mu_1, \mu_2),
\end{align*}
where the last inequality follows from the fact that the controlled dynamics with the same control $\theta$ preserve the distance between initial distributions (the noise is additive and independent of the state). 
Thus, $\mathcal{V}_\eps$ inherits the Lipschitz constant of $G$.

\begin{cor}[General Measure Error]\label{cor:general_measure}
Let $\mathbf{X} = (X_1, \dots, X_N)$ be a system of $N$ i.i.d. random variables with law $\mu$ in Theorem \ref{thm:value_convergence}. Then
\begin{equation}
    \E \left| \frac{v^N(t, \mathbf{X})}{N} - \mathcal{V}_\eps(t, \mu) \right| \le \frac{L_\eps}{2N} + L_G \cdot \E [\mathcal W_2(\mu^N_{\mathbf{X}}, \mu)]
\end{equation}
where $L_G$ is the Lipschitz constant of $G$. 
\end{cor}

\begin{rmk} \label{rmk:sampling_error}
The expectation of the Wasserstein distance $\E[\mathcal W_2(\mu^N_{\mathbf{X}}, \mu)]$ converges to zero as $N$  tends to infinity but with the rate depending on the dimension $d$ and the moments of $\mu$. Assuming $\mu$ has sufficiently many finite moments, Fournier and Guillin \cite{fournier2015rate} established the following rates:
$$
\E[\mathcal W_2(\mu^N_{\mathbf{X}}, \mu)] \le  \begin{cases}
C_{\alpha }N^{-1/2+\alpha}, \quad \text{for}\, \alpha \in (0, \frac{1}{2}), & \text{if } d \leq 4, \\
C_2N^{-1/d}, & \text{if } d > 4,
\end{cases}
$$
where $C_{\alpha} $ depends on the dimension $d$ and the $\alpha$, and $C_2$ depends on the dimension $d$.
\end{rmk}



\subsection{Regularization error}
The second source of error arises from replacing the hard minimization problem with a regularized mean-field stochastic 
control problem.

\begin{thm}[Regularization Error]\label{thm:regularization_error}
Let $G: \Pcal \to \bR$ be H\"older continuous with respect to the 2-Wasserstein distance, i.e., there exist $L_G>0$ and $\alpha \in (0,1]$ such that
$|G(\mu) - G(\nu)| \le L_G |\mathcal W_2(\mu, \nu) |^\alpha$ for all $\mu, \nu \in \Pcal$.
Assume that $G$ admits a global minimizer $\mu^*$ with a finite second moment (i.e., $\int |x|^2 d\mu^*(x) < \infty$).
Let $\mathcal V_0 = \inf_{\mu} G(\mu)$. Then the value function $\mathcal{V}_\eps$ satisfies
\begin{equation}
    0 \le \mathcal{V}_\eps(0,\delta_{x_0}) - \mathcal V_0 \le C \cdot \eps \ln\left(\frac{1}{\eps}\right), \quad \text{for } \eps \in (0,e^{-1}),
\end{equation}
for some constant $C > 0$ depending on $d$, $L_G$, $\alpha$, and the 2-Wasserstein  distance ($\mathcal W_2(\delta_{x_0}, \mu^*)$) between the Dirac measure at $x_0$ and  
$\mu^*$.
\end{thm}
 
\begin{proof}
\textbf{Step 1. } First, we reformulate the stochastic control problem as a variational problem over probability measures.
Let $\mathbb{P}$ be the Wiener measure for paths starting at $x_0$ at time 0 
(the law of the uncontrolled proces $X^{0,x_0;0}$ where $\theta \equiv 0$). 
For any admissible control $\theta \in \Theta$, let $\mathbb{Q}^\theta$ be the law of the controlled process $X^{0,x_0;\theta}$.

By Girsanov's theorem, the Radon-Nikodym derivative of $\mathbb{Q}^\theta$ with respect to $\mathbb{P}$ is
$$
\frac{d\mathbb{Q}^\theta}{d\mathbb{P}} = \exp\left( \int_0^1 \theta_t \cdot dW_t - \frac{1}{2} \int_0^1 |\theta_t|^2 dt \right).
$$
The relative entropy (Kullback-Leibler divergence) is given by
$$
\mathcal H (\mathbb{Q}^\theta | \mathbb{P}) = \E^{\mathbb{Q}^\theta} \left[ \ln \frac{d\mathbb{Q}^\theta}{d\mathbb{P}} \right] = \E \left[ \frac{1}{2} \int_0^1 |\theta_t|^2 dt \right].
$$
This identity allows us to rewrite the cost functional in terms of the induced measure $\mathbb{Q}^\theta$, and the control problem becomes
\begin{equation}
    \label{eq:variational}
    \mathcal{V}_\eps(0,\delta_{x_0}) = \inf_{\theta \in \Theta} \left\{ G(\mathcal{L}(X_1^{0,x_0;\theta})) + \eps \mathcal{H}(\mathbb{Q}^\theta | \mathbb{P}) \right\}.
\end{equation}

The lower bound is obvious from the control problem and alternatively, it follows immediately from the non-negativity of the relative entropy term,
 i.e., as for any admissible control $\theta$, $\mathcal H (\mathbb{Q}^\theta | \mathbb{P}) \ge 0$, we have
\begin{equation}
    \mathcal{V}_\eps(0,\delta_{x_0}) = \inf_{\theta} \left\{ G(\mathcal{L}(X_1^{0,x_0;\theta})) + \eps \mathcal H(\mathbb{Q}^\theta | \mathbb{P}) \right\} \ge \inf_{\theta} G(\mathcal{L}(X_1^{0,x_0;\theta})) \ge \inf_{\mu \in \Pcal} G(\mu) =\mathcal V_0.
\end{equation}

\textbf{Step 2.} Next, we prove the upper bound by constructing a specific suboptimal control strategy $\theta^\lambda$.
Let $\mu^*$ be a global minimizer of $G$, i.e., $G(\mu^*) = \mathcal V_0$.
By assumption, $\mu^*$ has a finite second moment. Note that $\mu^*$ may be singular (e.g., a Dirac measure). 
To handle this, we smooth $\mu^*$ by convolving it with a Gaussian kernel. 
We define a smoothed candidate measure $\mu^\lambda$ by convolving $\mu^*$ with 
a Gaussian kernel $\phi_\lambda$ of variance $\lambda \in (0,1)$, i.e., 
\begin{equation}
    \mu^\lambda(dx) = (\mu^* \ast \phi_\lambda)(x) dx = \int_{\bR^d} \frac{1}{(2\pi\lambda)^{d/2}} \exp\left( -\frac{|x - y|^2}{2\lambda} \right) \, \mu^*(dy) \, dx.
\end{equation}

The theory of Schrödinger bridges (see, e.g., \cite{follmer2006random,leonard2014survey}) guarantees the existence of such a control $\theta^\lambda$
that the law of the controlled process at time $t=1$ matches $\mu^\lambda$, 
 i.e., $\mathcal{L}(X_1^{0,x_0;\theta^\lambda}) = \mu^\lambda$.
Moreover, 
the minimum entropy cost to achieve a marginal distribution $\mu^\lambda$ at time $t=1$ is simply the relative entropy of the marginals, i.e., 
$$
\inf_{\theta: \mathcal{L}(X_1^{0,x_0;\theta}) = \mu^\lambda} \mathcal H (\mathbb{Q}^\theta | \mathbb{P}) = \mathcal H (\mu^\lambda | \gamma_{x_0}),
$$
where $\gamma_{x_0} = \mathcal{L}(W_1+x_0) \sim \mathcal{N}(x_0, I_d)$. 
 Substituting this control into the cost functional gives an upper bound
\begin{equation}
    \mathcal{V}_\eps(0,\delta_{x_0}) \le   G(\mu^\lambda) + \eps \mathcal H(\mu^\lambda | \gamma_{x_0}).
\end{equation}

As $G$ is  H\"older continuous with respect to $\mathcal W_2$, we have
$
|G(\mu^\lambda) - G(\mu^*)| \le L_G \cdot |\mathcal  W_2(\mu^\lambda, \mu^*)|^\alpha.
$
Recalling that the measure $\mu^\lambda$ is obtained by adding independent noise $\sqrt{\lambda}Z$ to $\mu^*$, with $Z \sim \mathcal{N}(0, I_d)$. 
The 2-Wasserstein distance is bounded by the second moment of the noise. i.e.,
$$
\mathcal W_2(\mu^\lambda, \mu^*) \le \sqrt{\E[|\sqrt{\lambda}Z|^2]} = \sqrt{d \lambda}.
$$
Thus, it follows that 
$$ G(\mu^\lambda) \le \mathcal V_0 + L_G (d \lambda)^{\alpha/2}. $$

For the relative entropy term, we decompose it as follows:
$$
\mathcal H(\mu^\lambda | \gamma_{x_0}) = \int_{\bR^d}  \ln \left(\frac{\mu^\lambda(dx)}{\gamma_{x_0}(dx)}\right) \mu^\lambda(dx) 
= -h(\mu^\lambda) - \int_{\bR^d}  \ln \left(\frac{\gamma_{x_0} (dx)}{dx}\right) \mu^\lambda(dx),
$$
where $h(\mu^\lambda)= - \int_{\bR^d}  \ln \left(\frac{\mu^\lambda(dx)}{dx}\right) \mu^\lambda(dx)$ is the differential entropy. On the one hand,  as conditioning
reduces entropy, we have $h(\mu^\lambda) \ge h(\gamma_{0}^\lambda)$, where $\gamma_0^\lambda = \mathcal{N}(0, \lambda I_d)$ is the Gaussian with mean 0 and covariance $\lambda I_d$,
 which yields the upper bound
$$-h(\mu^\lambda) \le -h(\gamma_0^\lambda) = -\frac{d}{2} \ln(2\pi e \lambda).$$
On the other hand, we estimate the cross-term
$$-\int_{\bR^d} \ln \left(\frac{\gamma_{x_0} (dx)}{dx}\right) \mu^\lambda(dx)
 = \int_{\bR^d}  \left( \frac{|x - x_0 |^2}{2} + \frac{d}{2} \ln(2\pi) \right) \mu^\lambda(dx) \le M_2(x_0) +d\lambda + \frac{d}{2} \ln(2\pi),$$
where $M_2(x_0) = \int |x - x_0|^2 d\mu^*(x)$ is nothing but $|\mathcal W_2(\delta_{x_0}, \mu^*)|^2$. 

Therefore, combining these estimates, we have
\begin{equation}
    \mathcal H(\mu^\lambda | \gamma_{x_0} ) \le C_1 + \frac{d}{2} \ln\left(\frac{1}{\lambda}\right)
\end{equation}
for some constant $C_1$ depending on $d$ and  $\mathcal W_2(\delta_{x_0}, \mu^*)$.

\textbf{Step 3.}
Substituting the estimates back into the upper bound for $\mathcal{V}_\eps(0,\delta_{x_0})$ gives 
$$
\mathcal{V}_\eps(0,\delta_{x_0}) - \mathcal V_0 \le L_G(d\lambda)^{\alpha/2} + \eps \left( C_1 + \frac{d}{2} \ln\left(\frac{1}{\lambda}\right) \right).
$$
Setting $\lambda =  \frac{\eps^{2/\alpha}}{d}$ yields that
\begin{equation*}
    \mathcal{V}_\eps(0,\delta_{x_0}) - \mathcal V_0 \le C  \eps \ln \left(\frac{1}{\eps}\right),  
\end{equation*}
where the constant $C$ depends on $d$, $L_G$, $\alpha$, and  $\mathcal W_2(\delta_{x_0}, \mu^*)$.
\end{proof}

\subsection{Numerical approximation with examples}

Theorem~\ref{thm:regularization_error} establishes that, for sufficiently small $\eps$, the regularized mean-field control (MFC) problem provides a consistent approximation of the original optimization problem over $\Pcal$. 
From a computational perspective, one may then seek approximate solutions to the regularized MFC problem via numerical approaches based on the associated McKean--Vlasov FBSDE, or alternatively through its equivalent formulation as a mean-field potential game; see \cite{carmona2018probabilistic}. 
Related numerical methodologies for mean-field control and mean-field game problems can be found, for instance, in \cite{achdou2020mean,carmona2022convergence,pfeiffer2017numerical,ruthotto2020machine} and the references therein.

Here, we focus on the numerical approximation of the $N$-particle control problem \eqref{reg-sto-control-prob-mf-N} as an approximation to the regularized MFC problem.
Observe that the $N$-particle control problem \eqref{reg-sto-control-prob-mf-N} constitutes a finite-dimensional stochastic control problem on $\bR^{dN}$, in the same spirit as the setting studied in Section~\ref{section:opt-finite-dim}. 
Consequently, numerical approximations can be developed analogously to those in Section~\ref{subsection:numerics_finite_dim}, 
by employing time discretization of the controlled SDEs with Monte Carlo simulations.

Again, we adopt the Euler-Maruyama scheme with a time step $\Delta t = 1/M$ for some integer $M$. 
For the time grid $t_k = k \Delta t$, $k=0, \dots, M$, in view of the optimal control \eqref{optimal-control-particle} and the discretization \eqref{Euler-scheme-1}, the discretized dynamics for each particle $i=1, \dots, N$ is given by:
\begin{equation}
  X^i_{0}=x_0,
  \quad  X^{i}_{t_{k+1}} = X^{i}_{t_k} + \left(\beta^i(t_k, X^{i}_{t_k})-X^{i}_{t_k}\right) \frac{\Delta t}{1-t_k} + \sqrt{\Delta t} Z^{i}_k, \label{Euler-scheme-particle}
\end{equation}
where $Z^{i}_k \sim \mathcal{N}(0, I_d)$ are independent standard normal random variables, and 
\begin{align}
    \beta^i\!\left(t_k,X_{t_k}^i\right)
    &= \frac{\mathbb E\!\left[ \exp\!\left(-\frac{N}{\eps}G( \frac{1}{N} \sum_{j=1}^N\delta_{\{\sqrt{1-t_k} \xi^j+x^j\}})\right)\,( \sqrt{1-t_k} \xi^i+x^i) \right]}
    {\mathbb E\!\left[ \exp\!\left(-\frac{N}{\eps}G(\frac{1}{N} \sum_{j=1}^N\delta_{\{\sqrt{1-t_k} \xi^j+x^j\}})\right) \right]}
    \bigg|_{x^j=X_{t_k}^j, \,j=1,\dots,N}\nonumber \\
    &\approx
    \sum_{l=1}^S   \frac{\exp\!\left(-\frac{N}{\eps}G( \frac{1}{N} \sum_{j=1}^N\delta_{\{\sqrt{1-t_k} \xi_{j,l}+X_{t_k}^j\}})\right) ( \sqrt{1-t_k} \xi_{i,l}+X_{t_k}^i)}
    {
    \sum_{l=1}^S \exp\!\left(-\frac{N}{\eps}G(\frac{1}{N} \sum_{j=1}^N\delta_{\{\sqrt{1-t_k} \xi_{j,l}+X_{t_k}^j\}})\right)}  \label{beta}
    \, .
\end{align}
Here, $\{\xi^j\}_{j=1}^N$ are $N$-independent standard normal random variables in \(\mathbb{R}^d\), 
and for each fixed $j$, $\{\xi_{j,l}\}_{l=1}^S$ are $S$ independent samples from the standard normal distribution $\mathcal{N}(0, I_d)$. 
It is different from the finite-dimensional case \eqref{beta-finite-dim} in that:
\begin{enumerate}
    \item[(i)] for each particle $i$, the empirical measure $\frac{1}{N}\sum_{j=1}^N \delta_{ X^j_{t_k}}$ is used in the computation of $\beta^i$, reflecting the mean-field dependence;
    \item[(ii)] while the procedure may be repeated over multiple outer iterations, 
    the re-initialization at each iteration is taken directly from the terminal particle configuration of the preceding iteration 
    (without an additional averaging/coupling step), since the objective is to approximate an optimizing \emph{distribution} rather than a single point estimator.
\end{enumerate}

The main computational difficulty is that the global cloud weight 
in \eqref{beta} depends on the full random terminal configuration
 $\frac1N\sum_{j=1}^N\delta_{X^j_{t_k} + \sqrt{1-t_k} \xi^{j}}$. 
 Thus, at every time step and for every Monte Carlo sample, 
 one must evaluate $G$ on an $N$-particle empirical measure, 
 and the resulting Gibbs factor 
 $\exp\!\left(-\frac{N}{\eps}G\!\left(\frac1N\sum_{j=1}^N\delta_{X^j_{t_k}+ \sqrt{1-t_k} \xi^{j}}\right)\right)$ 
 accumulates fluctuations from all particles simultaneously. 
 When $N$ is large or $\eps$ is small, this produces 
 severe weight concentration 
 and a substantial increase in computational cost, especially for 
 interacting functionals $G$ whose evaluation already scales 
 superlinearly in $N$.

For a general functional $G:\Pcal\to\bR$, a natural particlewise localization 
is obtained by freezing the background empirical measure 
and varying only one particle. 
More precisely, for a background measure $\nu\in\Pcal$ and 
a candidate position $y\in\bR^d$, define the particlewise effective cost
\begin{equation}\label{NG-def}
    \mathcal G_N(y;\nu) := N\,G\!\left(\frac1N\delta_y + \frac{N-1}{N}\nu\right).
\end{equation}
Assume that $G$ admits a linear functional derivative 
$\frac{\delta G}{\delta \mu}(\nu)(\cdot)$ 
and a second-order expansion at $\nu$ in the direction of signed perturbations. 
Since $\frac{\delta G}{\delta \mu}(\nu)(\cdot)$ is defined only up to an additive constant, 
the combination appearing below is invariant under the choice of representative. Writing
$    \eta_N^{y,\nu}=\frac1N(\delta_y-\nu)$,
so that $\frac1N\delta_y+\frac{N-1}{N}\nu=\nu+\eta_N^{y,\nu}$, we obtain formally
\begin{equation}\label{NG-expansion}
    \mathcal G_N(y;\nu)
    = N G(\nu) + \left[\frac{\delta G}{\delta \mu}(\nu)(y) - \int_{\bR^d} \frac{\delta G}{\delta \mu}(\nu)(z)\,\nu(dz)\right] + O\!\left(\frac1N\right).
\end{equation}
Indeed,
\begin{equation*}
    G(\nu+\eta_N^{y,\nu})
    = G(\nu) + \int_{\bR^d} \frac{\delta G}{\delta \mu}(\nu)(z)\,\eta_N^{y,\nu}(dz) 
    + O\!\left(\frac{1}{N^2}\right),
\end{equation*}
and multiplying by $N$ yields \eqref{NG-expansion}. 
Hence the dominant term $N G(\nu)$ is independent of $y$ and 
cancels in the ratio defining $\beta^i$, while the remaining 
fluctuation is only of order $O(1)$. This suggests particle-wise approximation
\begin{equation}
    \beta^i(t_k,X_{t_k}^i)
    \approx
    \frac{\mathbb E\!\left[\exp\!\left(-\frac1\eps\mathcal G_N\big(\sqrt{1-t_k}\,\xi^i+x^i;\, \eta^i \big)
    \right)\,\big(\sqrt{1-t_k}\,\xi^i+x^i\big)\right]}
    {\mathbb E\!\left[\exp\!\left(-\frac1\eps\mathcal G_N\big(\sqrt{1-t_k}\,\xi^i+x^i;\,\eta^i\big)\right)\right]}
    \bigg|_{x^i=X_{t_k}^i, \, \eta^i=\nu_{t_k}^{N,-i}},
\end{equation}
where $\nu_{t_k}^{N,-i}:=\frac{1}{N-1}\sum_{j\neq i}\delta_{\widehat Y^{(i)}_j}$ with $\widehat Y^{(i)}_j$ being a sample of $X^j_{t_k} + \sqrt{1-t_k}\,\xi^j$ at time $t_k$ associated to particle $i$, 
is the empirical background seen by particle $i$. 
Heuristically, this is a frozen-background or propagation-of-chaos approximation 
with the other particles being treated through their empirical law. 
Note that even when $G$ is not smooth enough for the first-order expansion \eqref{NG-expansion}, $\mathcal G_N(y;\nu)$ remains a natural expression
 for isolating the local contribution of one particle and the above particle-wise approximation works.

In practice, one often employs a broader localized family of approximations. 
First, the background measure $\nu_{t_k}^{N,-i}$ may be replaced by a 
frozen proxy $\widehat \nu_{t_k}^{N,-i}$, for instance the current empirical 
swarm, a noisy forward projection of that swarm, or another cheaply computed 
proxy of the future crowd. 
Second, when evaluating $\mathcal G_N(y;\widehat \nu_{t_k}^{N,-i})$ remains expensive, 
one may replace the full background measure by a subsampled empirical one, 
i.e.,
\begin{equation}
    \widehat \nu_{t_k}^{N,-i,m} := \frac1m \sum_{r=1}^m \delta_{\widehat Y_{J_r}^{(i)}},
\end{equation}
where $J_1,\dots,J_m$ are sampled indices from the global cloud and $\widehat Y_{J_r}^{(i)}$ denotes a sample of $X^{J_r}_{t_k} + \sqrt{1-t_k}\,\xi^{J_r}$ at time $t_k$, associated to particle $i$. 
Then we use the quantity
\begin{equation}
    \mathcal G_{N,m}(y;\widehat \nu_{t_k}^{N,-i}) := N\,G\!\left(\frac1N\delta_y + \frac{N-1}{N}\widehat \nu_{t_k}^{N,-i,m}\right),
\end{equation}
or, more generally, any computational proxy 
$\widetilde{\mathcal G}_{N,m}(y;\widehat \nu_{t_k}^{N,-i})$ 
to approximate the local contribution of particle $i$. 
The corresponding Monte Carlo update then takes the generic form
\begin{equation}\label{eq:beta-particle-wise-general}
    \beta^i(t_k,X_{t_k}^i)
    \approx
    \frac{\sum_{l=1}^S \exp\!\left(-\frac1\eps\widetilde{\mathcal G}_{N,m}(Y_{i,l};\widehat \nu_{t_k}^{N,-i})\right) Y_{i,l}}
    {\sum_{l=1}^S \exp\!\left(-\frac1\eps\widetilde{\mathcal G}_{N,m}(Y_{i,l};\widehat \nu_{t_k}^{N,-i})\right)},
\end{equation}
where $Y_{i,l}=X_{t_k}^i+\sqrt{1-t_k}\,\xi_{i,l}$. 
The case $m=N-1$ and $\widetilde{\mathcal G}_{N,m}=\mathcal G_N$ 
recovers the global cloud weight with a frozen background measure, 
while smaller $m$ yields a further variance-cost tradeoff through subsampling. 
The four numerical examples considered below are all of this localized type. 
They all replace the global cloud weight by a particle-wise one, 
and the global cloud is subsampled when evaluating the interaction term,
 conditionally in Example \ref{exam:2d_spring} 
and explicitly in Examples \ref{exam:2d_newtonian}, \ref{exam:double_hula_hoop} 
and \ref{exm:generative_modeling_single_row}.

This type of localization is consistent with several established ideas 
in the literature. In high-dimensional sequential Monte Carlo, 
local or block particle filters replace global weights by localized weights precisely to avoid weight collapse 
and the curse of dimensionality; see \cite{rebeschini2015local,snyder2008obstacles}. 
From a statistical viewpoint,
 the same structural principle underlies mean-field variational approximations 
and particle/coordinate-wise updates for Gibbs-type models; see \cite{jordan1999introduction}. 
In our setting, the particle-wise approximation and its subsampled variants play the analogous role of 
replacing the full interacting cloud weight
 by a one-particle effective weight against a frozen background measure.

For clarity, Algorithm \ref{algorithm:mean-field} below is presented 
in its exact global cloud weight form, which discretizes as in \eqref{beta}. 
This exact version serves as the conceptual reference scheme. 
In actual implementations, one may replace the global weight in \eqref{beta} by the particle-wise approximation
 \eqref{eq:beta-particle-wise-general} whenever the exact evaluation of $G$
  becomes prohibitively expensive or numerically unstable. 
 The pseudocode is presented in Algorithm \ref{algorithm:mean-field} which we refer to 
 as the Stochastic Control Method (SCM) for Optimization on $\Pcal$. 
The output is an empirical measure approximating the optimal distribution.

\begin{algorithm}[H]
\caption{Stochastic Control Method for Optimization on $\Pcal$}
\label{algorithm:mean-field}
\begin{algorithmic}[1]
\State \textbf{Input:} Particles $N$, Horizon $T$, Steps $M$, Reg. $\eps$, Samples $S$, Iterations $L$.
\State \textbf{Initialize:} $X_0^{(i)} \leftarrow x_{\text{start}}$, $\Delta t \leftarrow T/M$.
\For{$l = 1$ to $L$}
\For{$m = 0$ to $M-1$}
\State $t \leftarrow m \Delta t$, $\tau \leftarrow T - t$
\For{$j = 1$ to $S$}
\State Sample $\xi_{1,j},\dots,\xi_{N,j} \sim \mathcal{N}(0, I_d)$ independently
\State Set $Y_{r,j} \leftarrow X_t^{(r)} + \sqrt{\tau}\,\xi_{r,j}$, for $r=1,\dots,N$
\State Form $\mu_{T}^{(j)} \leftarrow \frac{1}{N}\sum_{r=1}^N \delta_{Y_{r,j}}$
\State Compute weight $w_j \leftarrow \exp\left(-\frac{N}{\eps}G\left(\mu_T^{(j)}\right)\right)$
\EndFor
\For{$i = 1$ to $N$}
\State Compute $\beta_i \leftarrow \frac{\sum_{j=1}^S w_j Y_{i,j}}{\sum_{j=1}^S w_j}$
\State Compute drift $\hat{\theta}_i \leftarrow \frac{\beta_i - X_t^{(i)}}{\tau}$
\State Update $X_{t+\Delta t}^{(i)} \leftarrow X_t^{(i)} + \hat{\theta}_i \Delta t + \sqrt{\Delta t} Z^{(i)}$, for $Z^{(i)}\sim \mathcal{N}(0, I)$
\EndFor
\EndFor
\State $X_0^{(i)} \leftarrow X_T^{(i)}$ \Comment{Recycle terminal state}
\EndFor
\State \textbf{Output:} $\mu^* \approx \frac{1}{N} \sum \delta_{X_T^{(i)}}$
\end{algorithmic}
\end{algorithm}

To bridge the gap between the idealized formulation in Algorithm \ref{algorithm:mean-field}
 and its numerical implementations, some other techniques are incorporated. 
 Although the pseudocode assumes a fixed time step $\Delta t = T/M$ and 
 a constant number of iterations $L$, practical applications may adopt adaptive, multi-phase scheduling, beginning with a standard time horizon ($T=1$) for global exploration
  and followed by refinement phases over shorter horizons (e.g., T=0.001) with finer time steps (e.g., $\Delta t=10^{-6}$),
   thereby enhancing stability during convergence to complex target distributions. 
   Furthermore, direct evaluation of exponential weights $\exp(-\mathcal{C}/\varepsilon)$
    is numerically infeasible for extremely small $\varepsilon$ (e.g., $10^{-299}$)
     due to floating-point underflow; this is mitigated by employing the Log-Sum-Exp technique as in Algorithm \ref{alg:particle_method}, 
     which performs computations in the logarithmic domain and normalizes exponents prior to exponentiation, 
     ensuring numerical stability in near-zero regularization regimes.

To illustrate the performance of Algorithm \ref{algorithm:mean-field}, we consider the following examples.

\begin{exam}[2D Newtonian Swarm (Circle Law)]\label{exam:2d_newtonian}
Consider the minimization of the energy functional associated with Newtonian interaction in 2D:
\begin{equation}
    G(\mu) =  \iint_{\bR^{2d}} \left(\frac{|x-y|^2}{2}-\ln(|x-y|)\right) \, d\mu(x) \, d\mu(y).
\end{equation}
 
In the numerical experiment reported in Figure~\ref{fig:2d_newtonian}, we take $N=200$ particles and discretize the time interval $[0,T]$ with $M=1000$ steps, where $T=1$. 
The regularization parameter is set to be $\eps=10^{-10}$, and the conditional expectations in \eqref{beta} are approximated 
using $S=100$ Monte Carlo samples. We perform a single outer iteration ($L=1$) and initialize all particles at the origin.

For this purely interaction-driven energy, the global minimizer is not unique but known explicitly: 
it is the uniform probability measure supported on a unit circle, with minimal energy equal to $0.75$, 
commonly referred to as the \emph{circle law} \cite{carrillo2020ellipse}. 
Starting from the initial distribution concentrated at the origin, the proposed SCM 
drives the particle system to expand radially and, by terminal time, produces an empirical distribution that closely approximates the uniform measure on the unit circle.
\begin{figure}[htbp]
    \centering
     \includegraphics[width=0.9\textwidth]{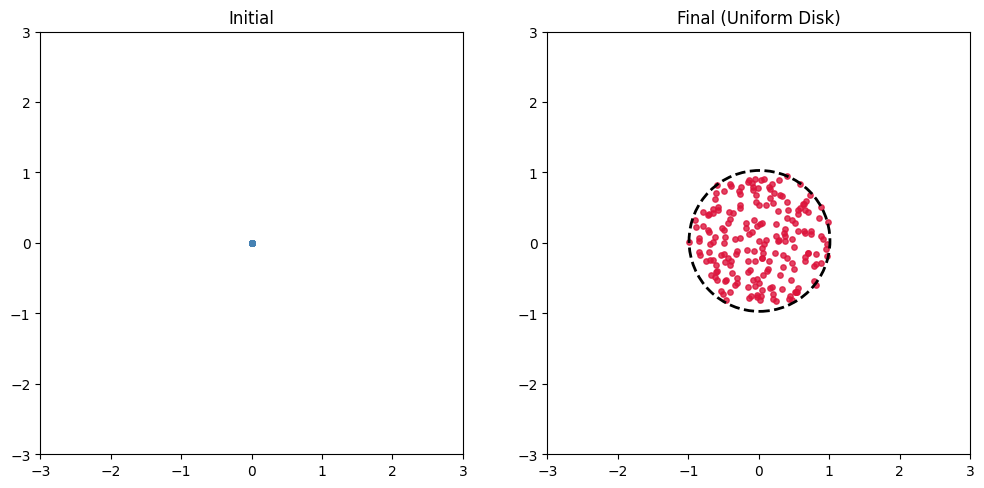}
    \caption{2D Newtonian Swarm Optimization. Left: Initial concentrated distribution at the origin. Right: Final distribution approximating the uniform measure on the unit circle, 
    the global minimizer of the Newtonian interaction energy.}
    \label{fig:2d_newtonian}
\end{figure}

To corroborate the convergence rate asserted in Theorem~\ref{thm:regularization_error}, 
we perform a numerical study in which the particle number $N$ is varied in $\{10, 20, 25, 50, 100, 125, 200, 250, 400, 500, 800, 1000\}$ and the corresponding Monte Carlo estimates of $\frac{1}{N} v^N_{\eps}(0, \mathbf{x})$ minus the minimum value ($=0.75$) of $G$ are plotted as a function of $1/N$; 
see Figure~\ref{fig:convergence_rate}. The resulting data exhibit an approximately linear dependence on $1/N$, and a least-squares fit provides strong empirical support for an $\mathcal{O}(1/N)$ finite-particle error. 
Moreover, extrapolating the fitted line to the limit $N\to\infty$ yields a vanishing asymptotic error.
\begin{figure}[htbp]
    \centering
     \includegraphics[width=0.6\textwidth]{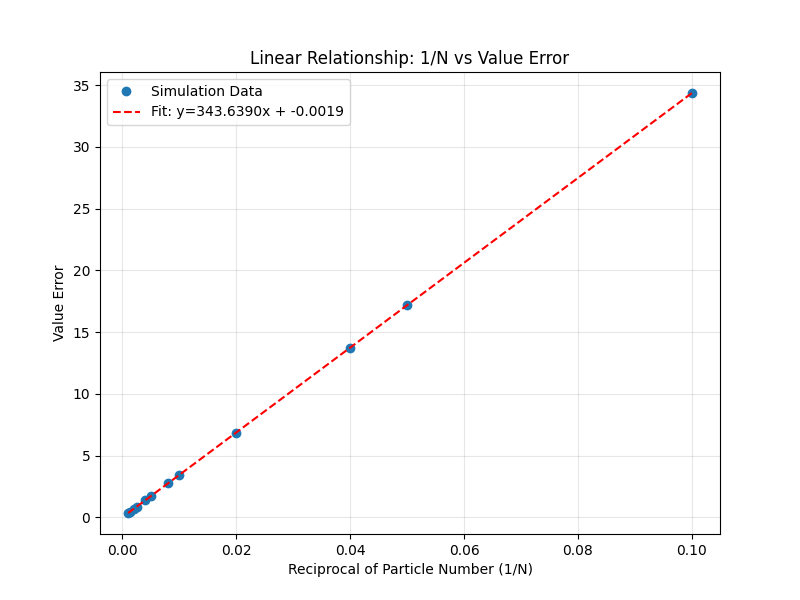}
    \caption{Convergence Rate of Value Errors. The plot of value error versus $\frac{1}{N}$ shows a linear trend, confirming the $O(1/N)$ convergence rate predicted by Theorem \ref{thm:value_convergence}.}
    \label{fig:convergence_rate}
\end{figure} 
\end{exam}  

In the above 2D Newtonian swarm example, it is not appropriate to study the sensitivity with respect to the regularization parameter $\eps$ because the particle approximation error also depends on $\eps$; see Theorem \ref{thm:value_convergence}. 
To isolate the regularization error, we consider the following example that satisfies the regularity, averaged convexity, and coercivity assumptions in Proposition \ref{prop-well-posedness-mfc} under which the constant $L_{\eps}$ in Thoerem \ref{thm:value_convergence} is independent of $\eps$.

\begin{exam}[2D Spring Swarm]\label{exam:2d_spring}
Consider the minimization of the following functional:
\begin{equation}
    G(\mu) =
    \frac{1}{2} \iint_{\mathbb{R}^2 \times \mathbb{R}^2} |x-y|^2 \, d\mu(x) \, d\mu(y).
    \end{equation}
    This quadratic interaction functional is minimized by any Dirac measure $\mu^*=\delta_{x^*}$, $x^*\in\bR^2$, and the minimal value is $G(\mu^*)=0$.

    To examine the impact of the regularization parameter $\eps$ on the numerical accuracy, we fix the particle number at $N=400$ and initialize all particles at the origin. 
    We discretize the time interval $[0,1]$ using a uniform step size $\Delta t=0.001$, and approximate the conditional expectations in \eqref{beta} with $S=100$ Monte Carlo samples at each time step. 
    We then consider $41$ values of $\eps$ evenly spaced on $[0.005,0.2]$, and for each $\eps$ compute a Monte Carlo estimate of the finite-particle value (at $t=0$) and its deviation from the minimal energy level $0$. 
    The resulting value errors are plotted against $\eps$ and $-\eps\ln(\eps)$ in Figure~\ref{fig:reg_error}. The observed scaling is consistent with the $\mathcal{O}(\eps\ln(1/\eps))$ regularization error predicted by Theorem~\ref{thm:regularization_error}.
\begin{figure}[htbp]
    \centering
     \includegraphics[width=0.9\textwidth]{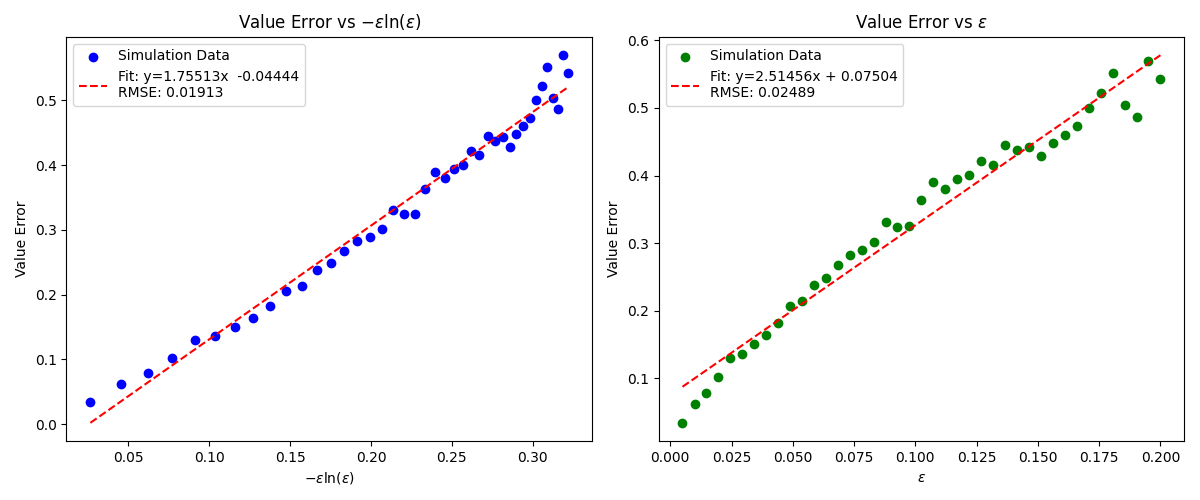}
    \caption{Regularization Error of value error. Left: Plot of value error versus $\eps \ln(1/\eps)$ shows a linear trend, supporting the $\mathcal{O}(\eps \ln(1/\eps))$ error estimate. Right: Plot of value error versus $\eps$. 
    The linear trend in the left panel provides a better ﬁt, as indicated by the lower
RMSE, corroborating the theoretical convergence rate established in Theorem \ref{thm:regularization_error}.}
    \label{fig:reg_error}
\end{figure}    
\end{exam}

In the above examples, all the particles are initialized at the same location, which facilitates the investigation of the convergence rate with respect to the particle number $N$ and the regularization parameter $\eps$. 
In the subsequent example, we consider a more general initialization with non-smooth functional and the minimizer may be a measure supported on disjoint sets.

\begin{exam}[Double Hula Hoop]\label{exam:double_hula_hoop}
We consider the minimization of the following functional:
\begin{equation}
    G(\mu) = {\int_{\mathbb{R}^2} V(x) \, d\mu(x)}
    + {\frac{1}{2} \iint_{\mathbb{R}^2 \times \mathbb{R}^2} \Phi(x-y) \, d\mu(x) \, d\mu(y)},
\end{equation}
where the confinement potential $V$ and interaction potential $\Phi$ are given by
\begin{align*}
V(x) &= \min \left( \frac{1}{2} \big( |x - c_L|^2 - R^2 \big)^2, \quad \frac{1}{2} \big( |x - c_R|^2 - R^2 \big)^2 \right)\, , \\
\Phi(x)&= - \ln(|x|),
\end{align*}
for $c_L = (-2,0)$, $c_R = (2,0)$, and $R=1$. The confinement potential $V$ has two wells centered at $c_L$ and $c_R$, each shaped like a hula hoop of radius $R$.
The interaction potential $\Phi $ is the Newtonian potential in 2D, which encourages particles to spread out. By theories of equilibrium measures in \cite{saff1997logarithmic}, 
the minimizer forms two disjoint, thickened annular bands with density increasing from the inner to the outer rim, achieving a global minimum energy where the 
outward pressure of logarithmic repulsion exactly balances the inward force of the quartic confinement; as the confinement potential is not smooth, it is not straightforward to derive an explicit formula for the density profile. 
In the numerical experiment reported in Figure~\ref{fig:double_hula_hoop}, we take $N=400$ particles and discretize the time interval $[0,1]$ with $M=1000$ steps. 
The regularization parameter is set to $\eps=10^{-299}$, and the conditional expectations in \eqref{beta} are approximated 
using $S=100$ Monte Carlo samples. We perform a single outer iteration ($L=1$) and initialize all particles  with standard normal distribution $N(0,1)$ on both the x and y axes.
The proposed SCM 
successfully drives the particle system to approximate the minimizer supported  around the two circles by terminal time.
\begin{figure}[htbp]
    \centering
     \includegraphics[width=0.9\textwidth]{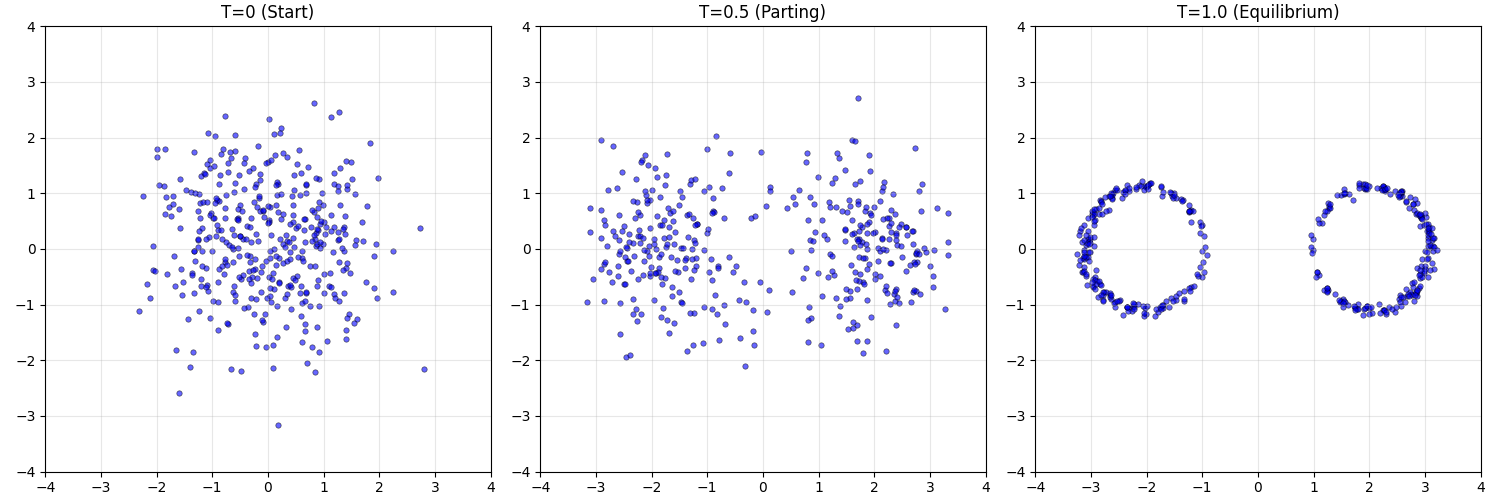}
    \caption{Double Hula Hoop Optimization. Left: Initial Gaussian distribution. Middle: Particles parting at $t=0.5$.  Right: Final distribution approximating the minimizer supported around the two circles.}
    \label{fig:double_hula_hoop}
\end{figure}    
\end{exam}  

The above example further inpires us to apply the SCM to sampling and generative problems. 
In generative modeling and aritificial intelligence (AI),  a fundamental problem is: 
given an empirical dataset $\mathbf{A} = \{x^1, \dots, x^J\} \subset \mathbb{R}^d$ 
which defines an underlying distribution $\hat{\mu} = \frac{1}{J} \sum_{j=1}^J \delta_{x^j}$, 
how can one generate a new set of data $\mathbf{D} = \{y^1, \dots, y^N\}$ such that the distribution of $\mathbf{D}$ 
is equal to or close to that of $\mathbf{A}$? This is nothing but a measure-matching question, 
an optimization problem on space of probability measures with the objective functional 
defined as a distance or divergence between probability measures, such as the Wasserstein distance, maximum mean discrepancy, and the Kullback-Leibler divergence.
Note that $J$ and $N$ need not be equal, allowing for data upsampling or compression. 
\begin{exam}[Generative Modeling]\label{exm:generative_modeling_single_row}
Taking $d=2$, we consider the minimization of the following distance functional:
\begin{equation*}
G(\mu) = \int_{\mathbb{R}^2} \min_{x\in\mathbf{A}} |x-y|^2 \, \mu(dy) 
- \frac{\alpha}{2} 
\int_{\mathbb{R}^2\times\mathbb{R}^2} \ln(|x-y|^2) \, \mu(dx) \mu(dy),
\end{equation*}
where $\alpha = 3\times 10^{-9}$, the first term encourages the generated data to be close to the dataset $\mathbf{A}$, 
and the second term encourages the generated samples to spread out.  
In the numerical experiment reported in Figure~\ref{fig:generative_modeling_single_row}, 
the target dataset $\mathbf{A}$ is a two-separate-horse silhouette consisting of $12{,}000$ points, 
while the generation is initialized from a snake-shaped configuration with $100{,}000$ points. 
We employ $N=100{,}000$ particles and discretize the time interval $[0,1]$ using $M=10{,}000$ time steps. 
The regularization parameter is set to $\eps=10^{-15}$, and the conditional expectations in \eqref{beta} are approximated with $S=100$ Monte Carlo samples. 
We first run a single outer iteration ($L=1$), and then recycle the resulting terminal particle configuration for two additional refinement iterations carried out 
over a shorter time horizon of length $0.001$ with time step $dt=10^{-6}$. 
The proposed SCM drives the particle system from the initial snake-shaped distribution toward an approximation of the target two-separate-horse distribution, 
achieving a high-fidelity reconstruction with an improved resolution;
see Figure~\ref{fig:generative_modeling_single_row}.
\begin{figure}[htbp]
    \begin{tabular}{@{{\hspace{-0.01\textwidth}}}c@{\hspace{-0.05\textwidth}}c@{\hspace{-0.05\textwidth}}c@{\hspace{-0.05\textwidth}}c@{\hspace{-0.05\textwidth}}c@{}}
    \includegraphics[width=0.38\textwidth]{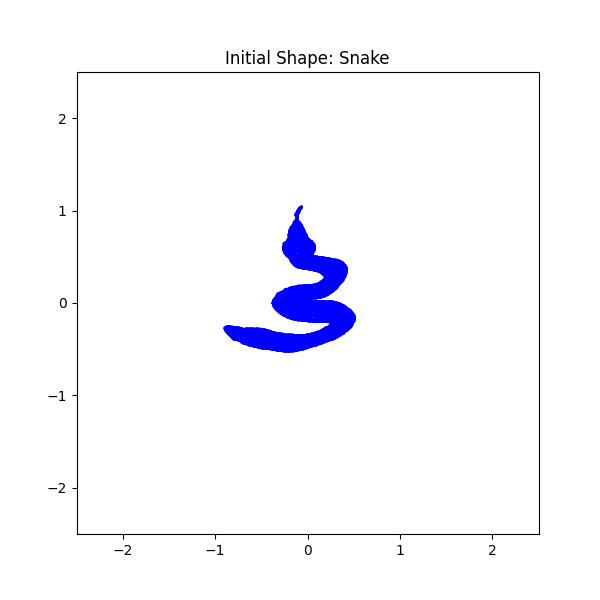} &
    \includegraphics[width=0.38\textwidth]{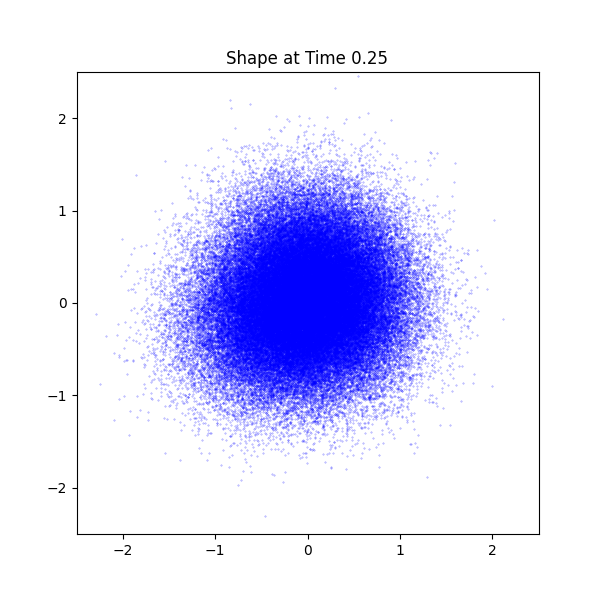} &
    \includegraphics[width=0.38\textwidth]{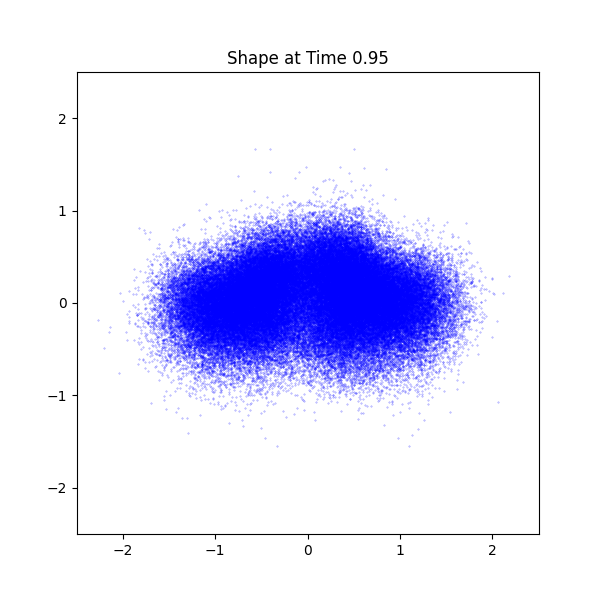} \\
    \includegraphics[width=0.38\textwidth]{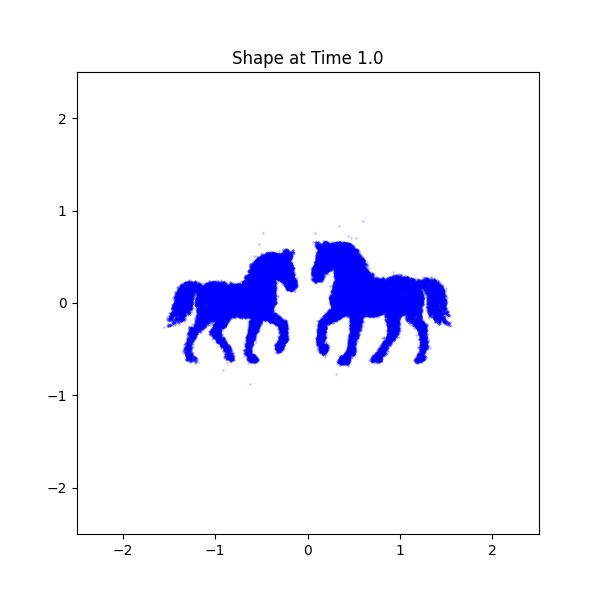} &
    \includegraphics[width=0.38\textwidth]{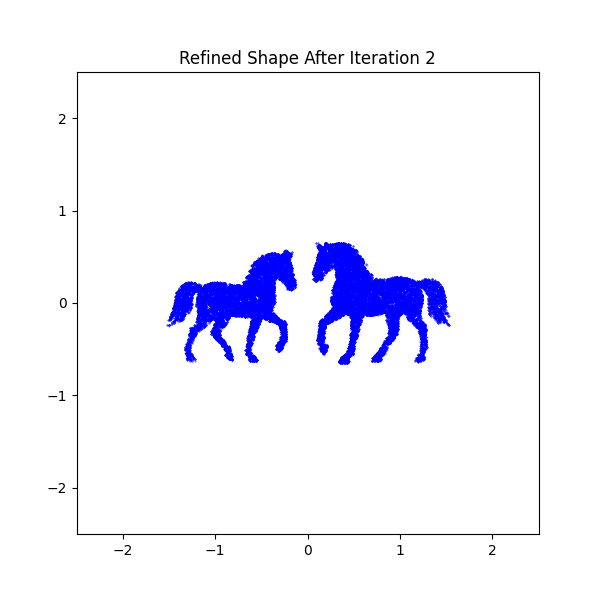} &
    \includegraphics[width=0.38\textwidth]{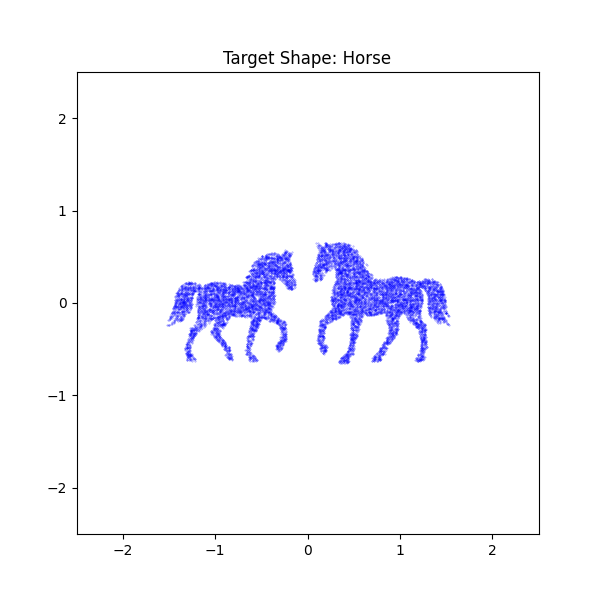} 
    \end{tabular}
    \caption{Generative Modeling. The three frames of the first row showing the evolution of the particle distribution from initial shape of a snake (left) through two intermediate steps at $t=0.25$ (middle) and $t=0.95$ (right). 
    The three frames of the second row showing the particle distribution at time $t=1.0$ (left), 
    the refined particle distribution (middle) with two iterations, 
    and the target empirical distribution (right) for comparison.}
    \label{fig:generative_modeling_single_row}
\end{figure}
\end{exam}
In view of the above example, 
the proposed SCM potentially provides a training-free and purely forward simulation-based alternative to diffusion models and Schr\"odinger bridge (SB) approaches for generative modeling. 
In contrast to diffusion models, which typically require an expensive offline training to learn a global score (see, e.g., \cite{yang2023diffusion}), 
SCM evaluates the optimal generative drift online through explicit probabilistic representations. 
Recent developments such as soft-constrained Schr\"odinger bridges \cite{garg2024soft,ma2025schrodinger} alleviate the numerical instability associated with hard terminal constraints in classical SB 
by introducing penalty-based McKean-Vlasov formulations; however, these methods remain fundamentally rooted in potential-based constructions 
(e.g., via Doob's $h$-transform) and often necessitate solving coupled forward-backward systems through global fixed-point iterations. 
By contrast, SCM avoids such forward-backward coupling. Indeed, by utilizing particle approximations together with the Cole-Hopf transformation,  
the Feynman-Kac representation, and the Bismut-Elworthy-Li formula, 
it yields an explicit representation of the generative drift and thereby reduces the computation to a single forward pass without training.

\section{Conclusion}\label{sec:conclusion}
This paper develops a stochastic-control-based framework for global optimization 
by casting the original problem as the vanishing-regularization limit of 
a family of stochastic control problems.  The objective functionals may be non-convex, non-differentiable, and defined on either finite-dimensional Euclidean spaces or the Wasserstein spaces of probability measures.
This formulation enables the use of dynamic programming methods and, in finite dimensions, the Cole-Hopf transformation to linearize the associated HJB equation.  
The resulting representations via Feynman-Kac and the Bismut-Elworthy-Li formula yield derivative-free expressions for both value functions and optimal controls.

For optimization over $\bR^d$, we establish convergence of the regularized control value to the global minimum at rate $\eps \ln(1/\eps)$ as $\eps\to0$.  
For optimization over Wasserstein space $\Pcal$, we introduce a regularized mean-field control formulation and approximate it by controlled $N$-particle systems.  
We prove quantitative convergence to the global optimum with rate
$
\frac{L_\eps}{N} + C\eps \ln\!\left(\frac{1}{\eps}\right),
$
as $N\to\infty$ and $\eps\to0$.

Based on the probabilistic representations, we further propose derivative-free numerical schemes and present computational experiments that support the theoretical findings.  
Reproducible code is available at:
\url{https://github.com/Jinnqiu/Stochastic-Control-Methods-for-Optimization}.
The numerical studies are intended as proof-of-concept illustrations rather than fully optimized implementations.  
A comprehensive computational study, including systematic parameter tuning, detailed benchmarking against established global optimization methods, and further applications in generative modeling and AI, is left for future work.

 \appendix
\section{Proof of Proposition \ref{prop-well-posedness-mfc}}\label{app:well-posedness_mfc}

\begin{proof}
Under the assumptions (a) (Regularity) and (b) (Averaged Convexity) in Proposition \ref{prop-well-posedness-mfc}, the  stochastic Pontryagin maximum principle with both necessary and sufficient conditions established in \cite[Theorems 4.5, 4.6 \& 5.1]{carmona2015forward} implies that, 
for each $(t,\mu)\in [0,1)\times \Pcal$, $\xi \sim \mu$, 
the optimal control $\theta^*$ exists uniquely, and the optimal state process $X^{t,\xi;\theta^*}$ satisfies uniquely the following FBSDE:
\begin{equation} \label{eq:fbsde}
    \begin{cases}
        dX^{t,\xi;\theta^*}_s = -\frac{1}{\eps} Y^{t,\xi;\theta^*}_s \, ds + dW_s, \quad X^{t,\xi;\theta^*}_t=\xi \sim \mu, \\
        dY^{t,\xi;\theta^*}_s = Z^{t,\xi;\theta^*}_s \, dW_s, \quad Y^{t,\xi;\theta^*}_1 = \partial_\mu G(\mathcal{L}(X^{t,\xi;\theta^*}_1))(X^{t,\xi;\theta^*}_1),\\
        \theta^*_s = - \frac{1}{\eps} Y^{t,\xi;\theta^*}_s, \quad \forall\, s\in[t,1],
    \end{cases}
\end{equation}
which, combined with the theories in \cite[Section 5.3 and Theorem 5.7]{chassagneux2014probabilistic}, implies that the value function $\mathcal V_{\eps}$ is a classical solution to the master equation \eqref{Master-eq} satisfying Assumption \ref{ass:smoothness} and that 
there holds the decoupling relation $Y^{t,\xi;\theta^*}_s = \partial_{\mu}\mathcal V_{\eps}(s, \mathcal L(X^{t,\xi;\theta^*}_s))(X^{t,\xi;\theta^*}_s)$ for all $s\in[t,1]$.  By Lemma \ref{lem:uniqueness}, the classical solution is unique.

It remains to justify the $\eps$-uniform boundedness estimate \eqref{boundedness-Lions-derivative}. Inspired by \cite[Chapter 2, Volume II]{carmona2018probabilistic}, we may derive the higher-order regularity of $\mathcal V_{\eps}$
 by differentiating FBSDE \eqref{eq:fbsde}. 

Let $\xi \sim \mu$. First,  differentiating the FBSDE \eqref{eq:fbsde} with respect to the spatial initial condition $x$ gives (cf.\cite[Eq (5.65), Volume II]{carmona2018probabilistic})
\begin{equation*}
    \begin{cases}
        d\big(\partial_x X^{t,x;\theta^*}_s\big)
        = -\dfrac{1}{\eps}\,\partial_x Y^{t,x;\theta^*}_s\,ds,
        \qquad \partial_x X^{t,x;\theta^*}_t = I_d, \\[0.6em]
        d\big(\partial_x Y^{t,x;\theta^*}_s\big)
        = \partial_x Z^{t,x;\theta^*}_s\, dW_s, \qquad
        \partial_x Y^{t,x;\theta^*}_1
        = \partial_x \partial_\mu G\!\left(\mathcal{L}\!\left(X^{t,\xi;\theta^*}_1\right)\right)\!\left(X^{t,x,\xi;\theta^*}_1\right)\,
        \partial_x X^{t,x;\theta^*}_1 ,
    \end{cases}
\end{equation*}
where $X^{t,\xi;\theta^*}_1$ satisfies the original FBSDE \eqref{eq:fbsde} while $X_s^{t,x,\xi;\theta^*}$ together with the pair $(Y_s^{t,x,\xi;\theta^*},Z_s^{t,x,\xi;\theta^*})$ satisfies FBSDE
\begin{equation*}
    \begin{cases}
        dX^{t,x,\xi;\theta^*}_s = -\frac{1}{\eps} Y^{t,x,\xi;\theta^*}_s \, ds + dW_s, \quad X^{t,x,\xi;\theta^*}_t=x, \\
        dY^{t,x,\xi;\theta^*}_s = Z^{t,x,\xi;\theta^*}_s \, dW_s, 
        \quad Y^{t,x,\xi;\theta^*}_1 = \partial_\mu G(\mathcal{L}(X^{t,\xi;\theta^*}_1))(X^{t,x,\xi;\theta^*}_1).
    \end{cases}
\end{equation*}

Fix an arbitrary $h\in\bR^d$. Applying It\^o's formula to the scalar product $\langle \partial_x X^{t,x;\theta^*}_s, \partial_x Y_s^{t,x;\theta^*} \rangle$ and taking conditional expectation, we obtain
{\small
\begin{align}
    \langle \partial_x X^{t,x;\theta^*}_sh, \partial_x Y_s^{t,x;\theta^*} h\rangle 
    &= \mathbb{E} \Big[ \underbrace{\langle \partial_x X_1^{t,x;\theta^*}h,\,  \partial_x \partial_\mu G(\mathcal{L}(X^{t,\xi;\theta^*}_1))(X^{t,x,\xi;\theta^*}_1) \partial_x X_1^{t,x;\theta^*}h \rangle }_{\ge 0 \text{ by Coercivity}} 
    + \frac{1}{\eps}   \int_s^1 |\partial_x Y_r^{t,x;\theta^*} h|^2 \, dr \big|\mathcal F_s \Big] ,
    \nonumber\\
    &\geq 0, \text{ a.s. } \quad \forall\, s\in[t,1]. \label{eq:pos-inner-prod-x}
\end{align}
}
Applying the chain rule further yields
\begin{align}
    d(|\partial_x X^{t,x;\theta^*}_sh|^2) = 2 \langle \partial_x X_s^{t,x;\theta^*}h, d(\partial_x X_s^{t,x;\theta^*}) h\rangle = -\frac{2}{\eps} \langle \partial_x X_s^{t,x;\theta^*}h, \partial_x Y_s^{t,x;\theta^*} h\rangle \, ds 
    \leq 0. \label{ineq-leq-0}
\end{align}
Recall $\partial_x X^{t,x;\theta^*}_t = I_d$. Thus, it holds that $|\partial_x X_s^{t,x;\theta^*}|_{op} \leq |I_d|_{op} = 1$ for all $s\in[t,1]$. Looking at $\partial_xY^{t,x;\theta^*}$, we have
\begin{align*}
    |\partial_x\partial_{\mu} \mathcal{V}_{\eps}(t, \mu) (x)|_{op}
    =
        |\partial_x Y_t^{t,x;\theta^*}|_{op} & = \left|\E\left[  \partial_x (\partial_\mu G(\mathcal{L}(X^{t,\xi;\theta^*}_1))(X^{t,x,\xi;\theta^*}_1)) \cdot \partial_x X^{t,x;\theta^*}_1 \big|\mathcal F_t \right]\right|_{op} \\
        &\leq K_G \E[|\partial_x X^{t,x;\theta^*}_1|_{op} | \mathcal F_t]\\
        &\leq K_G ,
\end{align*}
which yields the desired bound for $\partial_x \partial_\mu \mathcal V_{\eps}$.

Let us differentiate the FBSDE \eqref{eq:fbsde} with respect to the measure $\mu$. 
Let $\phi \in L^2(\mathbb{R}^d, \mu; \mathbb{R}^d)$ be a vector field representing the direction of the perturbation. 
This induces a perturbation of the initial random variable $\xi$ given by $\delta X^\phi_t = \phi(\xi) \in L^2(\Omega;\mathbb{R}^d)$.
 The tangent process $(\delta X_s, \delta Y_s, \delta Z_s)$ is defined as the $L^2(\Omega)$-limit of the difference quotients:
{\small\[
        \delta X_s := \lim_{r \to 0} \frac{X_s^{t,\xi+r\phi(\xi);\theta^r} - X_s^{t,\xi;\theta^*}}{r}, \quad \delta Y_s := \lim_{r \to 0} \frac{Y_s^{t,\xi+r\phi(\xi);\theta^r} - Y_s^{t,\xi;\theta^*}}{r}, \delta Z_s := \lim_{r \to 0} \frac{Z_s^{t,\xi+r\phi(\xi);\theta^r} - Z_s^{t,\xi;\theta^*}}{r}   ,
    \]
} where $\theta^r$ is the optimal control associated to the initial condition $\xi + r \phi(\xi)$. Differentiating the FBSDE coefficients leads to the following linearized system  
    \begin{equation} \label{eq:tangent_lqr}
        \begin{cases}
            d(\delta X_s) = -\frac{1}{\eps} \delta Y_s \, ds, \quad \delta X_t = \phi(\xi), \\
            d(\delta Y_s) = \delta Z_s \, dW_s, \\
            \delta Y_1 = \partial_x \partial_\mu G(\mathcal{L}(X^{t,\xi;\theta^*}_1))(X^{t,\xi;\theta^*}_1) \cdot \delta X_1 + \tilde{\mathbb{E}}[\partial^2_{\mu\mu} G(\mathcal{L}(X^{t,\xi;\theta^*}_1))(X^{t,\xi;\theta^*}_1) (\tilde{X}^{t,\xi;\theta^*}_1) \cdot \delta \tilde{X}_1].
        \end{cases}
    \end{equation}
    Applying It\^o's formula to the scalar product $\langle \delta X_s, \delta Y_s \rangle$, we have
\begin{align*}
   d\big( \langle \delta X_s, \delta Y_s \rangle\big)
    &= - \frac{1}{\eps}  |\delta Y_s|^2 \,ds
    + \langle \delta X_s, \delta Z_s \, dW_s \rangle.
\end{align*}
To obtain an $\eps$-uniform estimate, we apply the chain rule to the process $s\mapsto |\delta X_s|$:
\begin{equation}\label{eq:chain_rule-L1}
    d|\delta X_s| = \frac{1_{|\delta X_s|>0}}{|\delta X_s|} \langle \delta X_s, d(\delta X_s) \rangle = -\frac{1_{|\delta X_s|>0}}{\eps|\delta X_s|}\langle \delta X_s,\delta Y_s\rangle\,ds,
\end{equation}
and further,
\begin{align}
    d\left(\frac{1_{|\delta X_s|>0}}{|\delta X_s|}\langle \delta X_s, \delta Y_s \rangle\right)  
    &
    =   \frac{1_{|\delta X_s|>0}}{|\delta X_s|}d\big(\langle \delta X_s, \delta Y_s \rangle\big) - \frac{1_{|\delta X_s|>0}}{|\delta X_s|^2} \langle \delta X_s, \delta Y_s\rangle d |\delta X_s| \nonumber \\ 
    &= 
    - \frac{1_{|\delta X_s|>0}}{\eps|\delta X_s|}  |\delta Y_s|^2 \,ds
    + \frac{1_{|\delta X_s|>0}}{|\delta X_s|} \langle \delta X_s, \delta Z_s \, dW_s \rangle
    + \frac{1_{|\delta X_s|>0}}{\eps|\delta X_s|^3} \langle \delta X_s, \delta Y_s\rangle^2 ds
    \nonumber\\
    &\leq  
         \frac{1_{|\delta X_s|>0}}{|\delta X_s|} \langle \delta X_s, \delta Z_s \, dW_s \rangle,
      \label{eq:chain_rule-L1-2}
\end{align}
where in the last inequality we have used the Cauchy-Schwarz inequality: $|\langle \delta X_s, \delta Y_s \rangle | \leq |\delta X_s| \cdot |\delta Y_s|$.
Therefore, the process $\left\{\frac{1_{|\delta X_s|>0}}{|\delta X_s|}\langle \delta X_s, \delta Y_s \rangle\right\}_{s\in[t,1]}$ is a supermartingale. 
Noticing that by the averaged convexity assumption on $G$, 
\begin{align*}
    \E\left[\frac{1_{|\delta X_1|>0}}{|\delta X_1|}\langle \delta X_1, \delta Y_1 \rangle\right] 
    &= \E\left[\frac{1_{|\delta X_1|>0}}{|\delta X_1|}\langle \delta X_1, \partial_x \partial_\mu G(\mathcal{L}(X^{t,\xi;\theta^*}_1))(X^{t,\xi;\theta^*}_1) \cdot \delta X_1\rangle\right]
    \\
    &\quad 
    + \E\left[\frac{1_{|\delta X_1|>0}}{|\delta X_1|}\Big\langle \delta X_1, \widetilde{\E}\big[\partial^2_{\mu\mu} G(\mathcal{L}(X^{t,\xi;\theta^*}_1))(X^{t,\xi;\theta^*}_1)(\widetilde X^{t,\xi;\theta^*}_1) \cdot \delta \widetilde X_1\big]\Big\rangle\right] 
    \\
    & \ge 0,
\end{align*}
we have by the supermartingale property that for all $s\in[t,1]$, 
\begin{equation*}
     \E\left[\frac{1_{|\delta X_s|>0}}{|\delta X_s|}\langle \delta X_s, \delta Y_s \rangle\right] \ge 0.
\end{equation*}
Then, integrating \eqref{eq:chain_rule-L1} from $t$ to $1$ and taking the expectations, we have
\begin{equation}\label{estimate-deltaX}
    \E\big[|\delta X_1|\big]  = \E\big[|\phi(\xi)|\big] -\frac{1}{\eps} \int_t^1 \E\left[\frac{1_{|\delta X_s|>0}}{|\delta X_s|}\langle \delta X_s,\delta Y_s\rangle\right]ds
    \le \E\big[|\phi(\xi)|\big].
\end{equation}

Then, we investigate the directional derivative of $\mathcal V_{\eps}$ in the direction $\phi$. 
In view of \cite[Eq. (5.59), Volume II]{carmona2018probabilistic}, denote the associated directional derivatives by $\delta_{\phi}X_s$, $\delta_{\phi}Y_s$, $\delta_{\phi}Z_s$, which satisfy the following FBSDE
 {\small   \begin{equation} \label{eq:deriv_mu_lqr}
        \begin{cases}
            d(\delta_{\phi} X_s) 
            = -\frac{1}{\eps} \delta_{\phi} Y_s \, ds, \quad \delta_{\phi} X_t = 0, \\
            d(\delta_{\phi} Y_s) = \delta_{\phi} Z_s \, dW_s, \\
            \delta_{\phi} Y_1 = 
            \partial_x \partial_\mu G(\mathcal{L}(X^{t,\xi;\theta^*}_1))(X^{t,x,\xi;\theta^*}_1) \cdot \delta_{\phi} X_1 
            + \tilde{\mathbb{E}}[\partial^2_{\mu\mu} G(\mathcal{L}(X^{t,\xi;\theta^*}_1))(X^{t,x,\xi;\theta^*}_1) (\tilde{X}^{t,\xi;\theta^*}_1) \cdot \delta \tilde{X}_1].
        \end{cases}
    \end{equation} 
 }
Applying It\^o's formula to $\langle \delta_{\phi} X_s, \delta_{\phi} Y_s \rangle$ and taking expectations give
\begin{align*}
    \E\left[\langle \delta_{\phi} X_1, \delta_{\phi} Y_1 \rangle\right] 
    &  
    = -\frac{1}{\eps} \E\left[\int_t^1 |\delta_{\phi} Y_s|^2 ds \right] + \E\left[ \langle \delta_{\phi} Y_t,\,\delta_{\phi} X_t\rangle \right] 
    = -\frac{1}{\eps} \E\left[\int_t^1 |\delta_{\phi} Y_s|^2 ds \right] \leq 0.
\end{align*}
On the other hand, the boundedness of $\partial_{\mu\mu}^2 G$ and the coercivity assumption yield that
{\small
\begin{align*}
    0\geq \E\left[\langle \delta_{\phi} X_1, \delta_{\phi} Y_1 \rangle\right] 
     &
    = \E\left[ \langle \delta_{\phi} X_1, \partial_x \partial_\mu G(\mathcal{L}(X^{t,\xi;\theta^*}_1))(X^{t,x,\xi;\theta^*}_1) \cdot \delta_{\phi} X_1 \rangle\right] 
    \\
    &\quad + \E\left[ \langle \delta_{\phi} X_1, \tilde{\mathbb{E}}[\partial^2_{\mu\mu} G(\mathcal{L}(X^{t,\xi;\theta^*}_1))(X^{t,x,\xi;\theta^*}_1) (\tilde{X}^{t,\xi;\theta^*}_1) \cdot \delta \tilde{X}_1] \rangle\right] 
    \\
    &\geq \kappa \, \E\left[ |\delta_{\phi} X_1|^2 \right] 
    - K_G \, \E\left[ |\delta_{\phi} X_1|\right] \cdot \tilde \E \left[|\delta \tilde{X}_1| \right],
\end{align*}
}
which together with \eqref{estimate-deltaX} implies that 
$\sqrt{\E\left[|\delta_{\phi} X_1|^2\right]} \leq \frac{K_G}{\kappa} \E\left[|\phi(\xi)|\right]$. Hence, by the BSDE for $\delta_{\phi} Y_t$, we have
\begin{align*}
    &|\E \left[\partial^2_{\mu\mu} \mathcal{V}_{\eps}(t,\mu)(x)(\xi) \cdot \phi(\xi) \right]|
    = |\E[\delta_{\phi} Y_t]| \\
    & = \left|\E\left[  \partial_x \partial_\mu G(\mathcal{L}(X^{t,\xi;\theta^*}_1))(X^{t,x,\xi;\theta^*}_1) \cdot \delta_{\phi} X_1 
        + \tilde{\mathbb{E}}[\partial^2_{\mu\mu} G(\mathcal{L}(X^{t,\xi;\theta^*}_1))(X^{t,x,\xi;\theta^*}_1)(\tilde{X}^{t,\xi;\theta^*}_1) \cdot \delta \tilde{X}_1] \big|\mathcal F_t \right]\right|\\
        &\leq K_G \E[|\delta_{\phi} X_1|  ] + K_G \tilde \E[|\delta \tilde{X}_1|  ] \\
        &\leq  \left(\frac{(K_G)^2}{\kappa} + K_G\right) 
        \E[|\phi(\xi)|] .
\end{align*}
In view of the arbitrariness of $\phi$, we obtain the desired bound for $\partial^2_{\mu\mu} \mathcal V_{\eps}$.
\end{proof}

 \section{Proof of Theorem \ref{thm:value_convergence}}
 \label{app:particle_convergence}
\begin{proof}[Proof of Theorem \ref{thm:value_convergence}]
 
The proof relies on the method of projection and the comparison principle for parabolic equations. 

\textbf{Step 1}
Define the projected function $w^N_{\eps}    : [0, 1] \times (\bR^d)^N \to \bR$ as the value of the mean-field problem scaled by $N$, evaluated at the empirical measure $\mu^N_{\mathbf{x}} = \frac{1}{N} \sum_{i=1}^N \delta_{x_i}$, i.e.,
\begin{equation}
    w^N_{\eps}(t, \mathbf{x}) := N \cdot \mathcal{V}_{\eps}\left(t, \frac{1}{N} \sum_{i=1}^N \delta_{x_i}\right), \quad \text{for } (t, \mathbf{x}) \in [0,1] \times (\bR^d)^N.
\end{equation}
We calculate the derivatives of $w^N_{\eps}$ with respect to the particle positions $x_i$ using the chain rule for Lions derivatives.
For each $i \in \{1, \dots, N\}$,
\begin{equation}
    \partial_{x_i} w^N_{\eps}(t, \mathbf{x}) = N \cdot \frac{1}{N} (\Lions \mathcal{V}_{\eps})(t, \mu^N_{\mathbf{x}})( x_i) 
    = \Lions \mathcal{V}_{\eps}(t, \mu^N_{\mathbf{x}})( x_i).
\end{equation}
For the second derivatives (Hessian), we have  
\begin{equation}
    \Delta_{x_i} w^N_{\eps}(t, \mathbf{x}) = \text{div}_x (\Lions \mathcal{V}_{\eps})(t, \mu^N_{\mathbf{x}})( x_i) 
    + \frac{1}{N} \text{Tr}\left( \partial^2_{\mu\mu} \mathcal{V}_{\eps}(t, \mu^N_{\mathbf{x}})( x_i)( x_i) \right),
\end{equation}
where $\partial^2_{\mu\mu} \mathcal{V}_{\eps}$ is the second-order Lions' derivative with respect to the measure.

\textbf{Step 2.}
Let us check how well the projected function $w^N_{\eps}$ satisfies the $N$-particle HJB equation. 
  Let $\mathcal{L}^N$ be the nonlinear HJB operator:
 $$
 \mathcal{L}^N[\phi] (t, \mathbf{x}) = -\partial_t \phi(t, \mathbf{x}) - \frac{1}{2} \sum_{i=1}^N \Delta_{x_i} \phi(t, \mathbf{x})
 + \frac{1}{2\eps} \sum_{i=1}^N |\partial_{x_i} \phi(t, \mathbf{x})|^2.
 $$
 Substituting $w^N_{\eps}$ into the operator gives
 \begin{align*}
     \mathcal{L}^N[w^N_{\eps}](t, \mathbf{x}) &= -N \partial_t \mathcal{V}_{\eps}(t, \mu^N_{\mathbf{x}}) 
     - \frac{1}{2} \sum_{i=1}^N \left( \text{div}_{x_i} (\Lions \mathcal{V}_{\eps})(t, \mu^N_{\mathbf{x}})(x_i) 
     + \frac{1}{N} \text{Tr}(\partial^2_{\mu\mu} \mathcal{V}_{\eps})(t, \mu^N_{\mathbf{x}})(x_i)(x_i) \right) \\
     &\quad + \frac{1}{2\eps} \sum_{i=1}^N |\Lions \mathcal{V}_{\eps}(t, \mu^N_{\mathbf{x}})(x_i)|^2.
 \end{align*}
 Rewrite the sums as integrals with respect to the empirical measure $\mu^N_{\mathbf{x}}$
 and factoring out $N$ gives
 {\small
 \begin{align*}
     \mathcal{L}^N[w^N_{\eps}](t, \mathbf{x}) 
     &= N \left[ -\partial_t \mathcal{V}_{\eps}(t, \mu^N_{\mathbf{x}}) 
     - \frac{1}{2} \int_{\mathbb R^d} \text{div}_y (\Lions \mathcal{V}_{\eps})(t, \mu^N_{\mathbf{x}})(y) \mu^N_{\mathbf{x}}(dy) 
     + \frac{1}{2\eps} \int_{\mathbb R^d} |\Lions \mathcal{V}_{\eps}(t, \mu^N_{\mathbf{x}})(y)|^2 \mu^N_{\mathbf{x}}(dy) \right] \\
     &\quad 
     - \frac{1}{2} \int_{\mathbb R^d} \text{Tr}(\partial^2_{\mu\mu} \mathcal{V}_{\eps})(t, \mu^N_{\mathbf{x}})(y, y) \mu^N_{\mathbf{x}}(dy).
 \end{align*}
 }
 Notice that the term in the square brackets $[ \dots ]$ is exactly the master equation evaluated at $\mu = \mu^N_{\mathbf{x}}$. 
 Since $\mathcal{V}_{\eps}$ is a solution to the master equation, this bracketed term is zero.
 The remaining term is the residual error
 \begin{equation}
     R^N_{\eps}(t, \mathbf{x}) = - \frac{1}{2} \int_{\mathbb R^d} \text{Tr}(\partial^2_{\mu\mu} \mathcal{V}_{\eps})(t, \mu^N_{\mathbf{x}})(y, y) \mu^N_{\mathbf{x}}(dy),
 \end{equation}
 whose boundedness depends on the regularity of the master equation solution $\mathcal{V}_{\eps}$ in Assumption \ref{ass:smoothness}.
 Thus, we have
$$
|R^N_{\eps}(t, \mathbf{x})| \le \frac{1}{2}  \left|\int_{\mathbb R^d}
\text{Tr}\left(\partial^2_{\mu\mu} \mathcal{V}_{\eps}(t, \mu^N_{\mathbf{x}})(y, y)\right) \mu^N_{\mathbf{x}}(dy) \right| \le \frac{1}{2} L_{\eps}.
$$

It is straightforward to check that the function $v^N_{\eps}$ defined through the (inverse) Cole-Hopf transformation in \eqref{eq:feynman_kac_N} 
is a classical solution to the HJB equation \eqref{eq:particle_hjb}. 
Therefore, the equation satisfied by $w^N_{\eps}$ is
$ \mathcal{L}^N[w^N_{\eps}] = R^N_{\eps}   $,
where the residual $R^N_{\eps}$ is uniformly bounded in $N$ but may depend on $\eps$, while
the equation satisfied by the value function $v^N_{\eps}$ is
$ \mathcal{L}^N[v^N_{\eps}] = 0. $

\textbf{Step 3.}
Let $e^N_{\eps}(t, \mathbf{x}) = v^N_{\eps}(t, \mathbf{x}) - w^N_{\eps}(t, \mathbf{x})$. 
Subtracting the equations for $v^N_{\eps}$ and $w^N_{\eps}$ gives
\begin{align*}
    0 &= \mathcal{L}^N[v^N_{\eps}] - \mathcal{L}^N[w^N_{\eps}] + R^N_{\eps} \\
    &= -\partial_t e^N_{\eps} - \frac{1}{2} \sum_{i=1}^N \Delta_{x_i} e^N_{\eps} + \frac{1}{2\eps} \sum_{i=1}^N \left( |\partial_{x_i} v^N_{\eps}|^2 - |\partial_{x_i} w^N_{\eps}|^2 \right) + R^N_{\eps}.
\end{align*}
We linearize the quadratic term using the identity $|a|^2 - |b|^2 = (a+b) \cdot (a-b)$:
$$
|\partial_{x_i} v^N_{\eps}|^2 - |\partial_{x_i} w^N_{\eps}|^2 = (\partial_{x_i} v^N_{\eps} + \partial_{x_i} w^N_{\eps}) \cdot \partial_{x_i} e^N_{\eps}.
$$
Define the vector field $\mathbf{b}(t, \mathbf{x}) = (\mathbf{b}_1, \dots, \mathbf{b}_N)$ where $\mathbf{b}_i = -\frac{1}{2\eps} (\partial_{x_i} v^N_{\eps} + \partial_{x_i} w^N_{\eps})$.
The error equation becomes a linear parabolic PDE with drift $\mathbf{b}$ and source $R^N_{\eps}$:
\begin{align}
    \begin{cases}
    \partial_t e^N_{\eps} + \frac{1}{2} \sum_{i=1}^N \Delta_{x_i} e^N_{\eps} + \sum_{i=1}^N \mathbf{b}_i \cdot \partial_{x_i} e^N_{\eps} = R^N_{\eps},
    \\
    e^N_{\eps}(1, \mathbf{x}) = 0.
    \end{cases}
\end{align} 
By the standard maximum principle (or comparison principle) for parabolic equations or a straightforward application of the Feynman-Kac formula, 
the solution is bounded by the source term $R^N_{\eps}$, i.e.,
$$
\sup_{\mathbf{x}} |e^N_{\eps}(t, \mathbf{x})| \le \int_t^1 \sup_{\mathbf{x}} |R^N_{\eps}(s, \mathbf{x})| ds \le (1-t) \cdot \frac{L_{\eps}}{2},\quad t\in[0,1].
$$
In this way, we have established that  
$$ |v^N_{\eps}(t, \mathbf{x}) - w^N_{\eps}(t, \mathbf{x})| \le  \frac{L_{\eps}}{2}. $$

Finally, dividing by $N$ to reach the normalized estimate
$$
\left| \frac{v^N_{\eps}(t, \mathbf{x})}{N} - \frac{w^N_{\eps}(t, \mathbf{x})}{N} \right| \le \frac{L_{\eps}}{2N},
$$
and recalling that $\frac{w^N_{\eps}}{N} = \mathcal{V}_\eps(t, \mu^N_{\mathbf{x}})$, we obtain the desired estimate.
\end{proof}


{\footnotesize
\let\oldthebibliography\thebibliography
\renewcommand\thebibliography[1]{%
    \oldthebibliography{#1}%
    \setlength{\itemsep}{0pt}
    \setlength{\parskip}{0pt}
}

}


%
 
\end{document}